\documentclass[12pt]{article}
\usepackage[centertags]{amsmath}
\usepackage{amsfonts}
\usepackage{amssymb}
\usepackage{amsthm}
\usepackage{newlfont}

\newcommand{\comment}[1]{}

\newtheorem{theorem}{Theorem}

\newtheorem{rem}{Remark}

\newtheorem{prop}{Proposition}

\newtheorem{definition}{Definition}

\newtheorem{corollary}{Corollary}
\begin{document}

\title{Developments in  perfect simulation of Gibbs measures through a
new result for the extinction of Galton-Watson-like processes  }
\author{ Emilio De Santis \\
{\small Dipartimento di Matematica }\\
{\small Universit\`a "La Sapienza",  Roma, Italia}\\
{\small \texttt{desantis@mat.uniroma1.it}} \and Andrea Lissandrelli \\
{\small Dipartimento di Matematica }\\
{\small Universit\`a "La Sapienza",  Roma, Italia}\\
{\small \texttt{andrea.lissandrelli@gmail.com}}} \maketitle

\begin{abstract}
This paper deals with the problem of perfect sampling from a Gibbs
measure with infinite range interactions. We present some sufficient
conditions for the extinction of processes which are like
supermartingales when large values are taken. This result has deep
consequences on perfect simulation, showing that local modifications
on the interactions of a model do not affect simulability. We also
pose the question to optimize over a class of sequences of sets that
influence the sufficient condition for the perfect simulation of the
Gibbs measure. We completely solve this question both for the long
range Ising models and for the spin models with finite range
interactions.
\end{abstract}

\medskip

\medskip \noindent \textbf{Keywords:} Perfect Simulation, Stochastic Ordering, Gibbs
Measures, Ising Models, Galton-Watson Processes.

MSC 2000: 60K35, 82B20, 68U20, 60K35, 60J80.

\section{Introduction}\label{sec1}

In this paper we deal with the problem of perfect simulation of
Gibbs measures. The first algorithm of this kind was realized by
\cite{PW}. This paper opened a new field of research which is
evolving in different directions. In \cite{MG1998}, the authors
extended the results of \cite{PW} to  continuous state space. In
\cite{HS00}, the study of perfect sampling from a Gibbs measure
started and in \cite{DP} the authors showed the importance of
percolation in perfect simulation algorithms for Gibbs measures with
finite range interactions. In \cite{CFF}, the authors dealt with
long memory processes which means  that the state of the process at
a fixed time depends  on all its past history. In \cite{GLO}, the
authors considered the problem of perfect sampling from a Gibbs
measure with infinite range interactions.

We start from the  paper \cite{GLO} and we pose new questions. The
algorithm  described in \cite{GLO} is based on a probability
distribution that we improve. It, in our paper, depends on the
choice of a sequence of growing sets  having appropriate properties.
In Section \ref{sec3}, we pose the question to optimize over this
sequence. We completely solve the problem in the case of finite
range interactions and in the case of infinite range Ising models
(see Theorem \ref{ThIsing} and Remark \ref{esplicito}). In Theorem
\ref{minimo}, we show that there always exists an optimal choice
that in general one is not able to calculate. In Theorem
\ref{ThIsing}, specialized for the Ising model, we make explicit the
best sequence of these growing sets.

In Section \ref{sec4}, we present some sufficient conditions for the
extinction of a discrete process with values in $\mathbb{N}$.
Theorem~\ref{prop:1} presents this result and it has applications in
various areas. The assumptions of Theorem~\ref{prop:1} are weaker
than the ones for the extinction of Galton-Watson process  which is
solved as a particular case (see \cite{Wi} for Galton-Watson
process). This result has implications for the perfect simulation
algorithm, see Theorem \ref{affonda}, because it supplies a weaker
sufficient condition for the applicability of the algorithm, than
the condition given in \cite{GLO}. Finally, we establish an
equivalence relation among interactions in the sense that two
interactions are equivalent if they only differ on a finite region.
By Theorem \ref{pr3}, we prove that, given two equivalent
interactions, if one respects the sufficient condition for the
perfect sampling, then the other one satisfies it too.

In Appendix \ref{App}, we provide the pseudo code of the algorithm,
for the Ising model, which calculates the optimal sequence of
growing sets and, at the same time, builds a perfect sampling from
the Gibbs measure observed on  a finite window.

\section{Synopsis}\label{sec2}


Let $S=\{-1,1\}^{\mathbb{Z}^d}$ be the set of spin configurations.
We endow $S$ with $\mathcal{S}$, the $\sigma$-algebra generated by
cylinders. A point $v\in\mathbb{Z}^d$ is called \emph{vertex}. Let
$\sigma(v)\in\{-1,1\}$ be the value of the configuration $\sigma \in
S$ at vertex $v\in\mathbb{Z}^d$, and let $\sigma^v\in\{-1,1\}$ be
the value of the configuration modified in $v$, i.e.
$$
\sigma^v(u)=\sigma(u)\hbox{ for all } u\neq v,\ \sigma^v(v)=-\sigma(v).
$$
We write $ A \Subset \mathbb{Z}^d$ to denote that $A$ is a finite
subset of $ \mathbb{Z}^d$. The cardinality of a set $A$ is indicated
with $|A|$. An \emph {interaction} is a collection of real numbers
$\mathbf{J}=\{J_B\in\mathbb{R}: B {\Subset} \mathbb{Z}^d, \,
|B|\geq2\}$   such that
\begin{equation}\label{sommab}
\sup_{v\in\mathbb{Z}^d }\sum_{B:v \in B}|B|   |J_B|<\infty.
\end{equation}
We denote by $\mathcal{J}$ the collection of all the interactions.
Note that in literature more general definitions of interactions are
considered  but in  our paper we will  only use this more
restrictive definition, as  done also in \cite {GLO}.

For brevity of notation set $\chi_B(\sigma)=\prod_{v\in B}\sigma(v)$
for any $B\Subset\mathbb{Z}^d$ and $\sigma\in S$.
 A probability measure $\pi$ on $(S,\mathcal{S})$ is said to be a
Gibbs measure relative to the interaction $\mathbf{J}\in\mathcal{J}$
if for all $v\in\mathbb{Z}^d$ and for any $\zeta\in S$
\begin{equation}\label{spec}
    \pi(\sigma(v)=\zeta(v)|\sigma(u)=\zeta(u)\ \forall u\neq v ) =
\frac{1} {{1+\exp({-2\sum_{B:v\in B} {(J_B \chi_B(\sigma)} } ))}}
\,\,\, a.s.
\end{equation}
which are called local specifications.

 \comment{Usually the Gibbs measure is written with an
interaction depending on a positive parameter $\beta$, i.e. $
J^{\beta}_B = \beta J_B $ for $B\Subset\mathbb{Z}^d$. In the paper
we deal with Gibbs measures in regime of uniqueness and the Gibbs
measure associated to the interaction will be denoted omitting the
explicit dependence on $\beta$.}


Let us define the set $ \mathcal{A}_{v} =\{ B \Subset \mathbb{Z}^d:
v \in B , J_B \neq 0  \}$, for $v \in \mathbb{Z}^d$; the set $
\mathcal{A}_{v}$ is finite or countable, therefore we can write
$\mathcal{A}_{v} = \{ A_{i, v}: i < N_v+1  \} $ where $ N_v = |
\mathcal{A}_{v}| $. We now introduce a sequence of sets with
appropriate properties that will replace the balls with distance
$L^1$ used in \cite{GLO}.

 Let $\mathbf{B}_v=(B_v(k)\Subset \mathbb{Z}^d :k\in\mathbb{N})$, for $v \in
\mathbb{Z}^d$, be a sequence of finite subsets in $\mathbb{Z}^d$
such that
\begin{itemize}
 \item[1)] $B_v(0)=\{v\}$;
 \item[2)] $B_v(k) \subset B_v(k+1)$ and $ B_v(k+1) \setminus  B_v(k) \neq \emptyset $, for $k \in  \mathbb{N}$;
 \item[3)] $\bigcup_{k\in \mathbb{N}}B_v(k)\supset \bigcup_{A \in \mathcal{A}_v}A= \bigcup_{i < N_v+1}A_{i,v}$.
\end{itemize}
We denote by $\mathcal{B}_v$ the space of the sequences verifying
1), 2) and 3).

\comment{ In this section we will use a fixed choice of
$\mathbf{B}_v \in \mathcal{B}_v$ that sometimes will not be
explicitly written.}

 In \cite{GLO}  a perfect simulation algorithm for a
Gibbs measure $\pi$ with long range interaction is presented. It can
be divided into two steps: \emph{the backward sketch procedure} and
\emph{the forward spin procedure}. For the applicability of the
algorithm they only have to assume  a condition on the first part,
i.e. on the backward sketch procedure. The algorithm is defined
through a Glauber dynamics having $\pi$ as reversible measure. A
process $(\sigma_t(v),v\in\mathbb{Z}^d,t\in\mathbb{R})$ taking
values in $S$ and having such dynamics, will be constructed. For any
$v\in\mathbb{Z}^d$, $\sigma\in S$ and $\mathbf{J} \in \mathcal{J}$
let $c_{v, \mathbf{J}}(\sigma)$ be the rate at which the spin in $v$
flips when the system is in the configuration $\sigma$,
$$
c_{v, \mathbf{J}}(\sigma)=\exp\bigg(-\sum_{B:v\in
B}J_B\chi_B(\sigma)\bigg) .
$$
The generator $G_{\mathbf{J}}$ of the process is defined on cylinder
functions $f:S\to\mathbb{R}$ as follows
$$
G_{\mathbf{J}}f(\sigma)=\sum_{v\in\mathbb{Z}^d} c_{v,
\mathbf{J}}(\sigma)[f(\sigma^v)-f(\sigma)].
$$
Assumption (\ref{sommab})
implies the uniform boundedness of the rates $c_{v,
\mathbf{J}}(\sigma)$ with respect to $v$ and $\sigma$, and
$$
\sup_{v\in\mathbb{Z}^d}\sum_{u\in\mathbb{Z}^d}\sup_{\sigma\in
S}|c_{v, \mathbf{J}}(\sigma)-c_{v, \mathbf{J}}(\sigma^u)|<\infty.
$$
Hence, Theorem 3.9 of \cite{liggett} guarantees that
$G_{\mathbf{J}}$ is effectively the generator of a Markovian process
$(\sigma_t(v),v\in\mathbb{Z}^d,t\in\mathbb{R})$ having $\pi $ as
invariant measure.

The difficulty of dealing with a measure with long range interaction
is overcome through a decomposition of the rates $c_v(\sigma)$ as
a convex combination of local range rates.

\comment{ ==================================Let
$\{\Omega,\mathcal{F},P\}$ the probability space on which is defined
an independent family of uniform random variables
$\{U_n(v),n\in\mathbb{N},v\in\mathbb{Z}^d\}$. Processes that we
consider will be constructed as functions of this family. A perfect
sampling algorithm of the measure $\pi$ is a map
$$
F:[{0,1}]^{\mathbb{N}\times \mathbb{Z}^d}\to S, \hbox{ such that }
F(U_n(v),n\in\mathbb{N},v\in\mathbb{Z}^d) \hbox{ has distribution }
\pi.
$$
The perfect sampling algorithm $F$ of the measure $\pi$ is said to
stop almost surely after a finite number of steps if for every site
$v\in\mathbb{Z}^d$, there exists a  finite random subset
$F^v\subset\mathbb{Z}^d$ and a finite random variable
$N_{STOP}^v\geq1$ such that if
$$U'_n(w)=U_n(w)\quad\forall w \in F^v,\ \forall 1\leq n\leq N_{STOP}^v
$$
then
\begin{equation}
\label{fo}F(U_n(w),n\in\mathbb{N},w\in\mathbb{Z}^d)(v)=F(U'_n(w),n\in\mathbb{N},w\in\mathbb{Z}^d)(v).
\end{equation}
Observe that from (\ref{fo}) it follows that
$F(U_n(w),n\in\mathbb{N},w\in\mathbb{Z}^d)(v)$ depends only on
$U_n(w)$ for $n=1,...,N_{STOP}^v$.
=========================================}

To present the decomposition we define two probability
distributions. The first one selects a random region of dependence
and the second one updates the value of the spins. For
$v\in\mathbb{Z}^d$, $\mathbf{J} \in \mathcal{J}$, let
\begin{equation}\label{lambda}
    \lambda_{v,\mathbf{J}, \mathbf{B}_v}(k)=\left  \{
    \begin{array}{ll}
      \exp (-2\sum\nolimits_{B:v\in B}|J_B|) &{\hbox{ if }}k=0, \\
      \exp (-\sum\nolimits_{B:v\in B,B\not\subset B_v(1)}|J_B| )-
  \exp (-2\sum\nolimits_{B:v\in B}|J_B|) &{\hbox{ if }}k=1,  \\
  \exp(-\sum\nolimits_{B:v\in B,B\not\subset B_v(k)}|J_B|)-
  \exp (-\sum\nolimits_{B:v\in B,B\not\subset B_v(k-1)}|J_B|) & {\hbox{ if }}k\geq
  2.\\
    \end{array}
\right.
\end{equation}
 Note that, for  $v\in\mathbb{Z}^d$,
$(\lambda_{v,\mathbf{J}, \mathbf{B}_v}(k): k\in\mathbb{N})$ is a
probability distribution on $\mathbb{N}$ because of properties 1),
2) and 3) of $ \mathcal{B}_v$. \comment{For brevity of notation we
will omit the indices $\mathbf{J}$, $\mathbf{B}_v$, writing
$\lambda_{v} = \lambda_{v,\mathbf{J}, \mathbf{B}_v}$ when there is
no ambiguity.}

Moreover, for each $v\in\mathbb{Z}^d$, $\sigma\in S$ and $
\mathbf{J} \in \mathcal{J} $ let
$M_{v,\mathbf{J}}=2\exp(\sum_{B,v\in B}|{J_B}|)$,
\begin{equation}\label{probp0}
    p_{v,\mathbf{J}, \mathbf{B}_v}^{[0]}(1)=p_{v,\mathbf{J},
\mathbf{B}_v}^{[0]}(-1)=\frac{1}{2},
\end{equation}
\begin{equation}\label{probp1}
    p_{v,\mathbf{J},
\mathbf{B}_v}^{[1]}(-\sigma(v)|\sigma)=\frac{1}{M_{v,\mathbf{J}}}\frac{\exp(-\sum_{B:v\in
B,B\subset B_v(1)}J_B\chi_B(\sigma))-\exp(-\sum_{B:v\in B,B\subset
B_v(1)}|J_B|)} {1-\exp(-2\sum_{B:v\in B,B\subset
B_v(1)}|J_B|)\exp(-\sum_{B:v\in B,B\not\subset B_v(1)}|J_B|)},
\end{equation}
and for $k\geq2$
$$
p_{v,\mathbf{J}, \mathbf{B}_v}^{[k]}(-\sigma(v)|\sigma)=
\frac{\exp(-\sum_{B:v\in B,B\subset
B_v(k-1)}J_B\chi_B(\sigma))}{M_{v,\mathbf{J}}}\cdot$$
\begin{equation}\label{probp2}
    \cdot\frac{\exp(-\sum_{B:v\in B,B\subset B_v(k),B\not\subset
B_v(k-1)}J_B\chi_B(\sigma))- \exp(-\sum_{B:v\in B,B\subset
B_v(k),B\not\subset B_v(k-1)}|J_B|)} {1-\exp(-\sum_{B:v\in
B,B\subset B_v(k),B\not\subset B_v(k-1)}|J_B|)}.
\end{equation}
Finally set for any $k\geq1$
$$
p_{v,\mathbf{J},
\mathbf{B}_v}^{[k]}(\sigma(v)|\sigma)=1-p_{v,\mathbf{J},
\mathbf{B}_v}^{[k]}(-\sigma(v)|\sigma).
$$
It is possible with some calculations to prove that
$p_{v,\mathbf{J}, \mathbf{B}_v}^{[k]}\in [0,1] $, thus $
p_{v,\mathbf{J}, \mathbf{B}_v}^{[k]} (-\sigma(v)|\sigma)$ is a
probability distribution on $\{ -1,1\}$. The probabilities in
(\ref{probp0})-(\ref{probp2}) will be used in the forward spin
procedure.

 Notice that for each $a\in\{-1,1\}$, $p_{v,\mathbf{J}, \mathbf{B}_v}^{[0]}(a)$
does not depend on $v$ and that, by construction, for any $k\geq1$,
$p_{v,\mathbf{J}, \mathbf{B}_v}^{[k]}(-\sigma(v) |\sigma)$ depends
only on the restriction of the configuration $\sigma$ to the set
$B_v(k)$. This is an important property that links the backward
sketch procedure to the forward spin procedure.

The announced decomposition of the rates $c_{v, \mathbf{J}}(\sigma)
$ is stated in \cite{GLO} in the following proposition.

\begin{prop}\label{dech}
Under condition (\ref{sommab}),
the following decomposition holds for any $\sigma\in S$
\begin{equation}\label{21nov}
c_{v, \mathbf{J}}(\sigma)=M_{v,\mathbf{J}}\bigg[\frac{\lambda_{v,
\mathbf{J}, \mathbf{B}_v}(0)}
{2}+\sum_{k=1}^{\infty}\lambda_{v,\mathbf{J},
\mathbf{B}_v}(k)p_{v,\mathbf{J},
\mathbf{B}_v}^{[k]}(-\sigma(v)|\sigma)\bigg].
\end{equation}
\end{prop}

Now in \cite{GLO} there is a construction of an auxiliary process
that links the Glauber dynamics with the perfect sampling algorithm
through  decomposition (\ref{21nov}).

Later on, for brevity of notation, we will omit the indices
$\mathbf{J}$, $\mathbf{B}_v$ when there is no ambiguity. The
backward sketch procedure constructs a process that we are going to
define. Let $M_v$ be  the mass associated to each vertex $v$. Let
$(C_n)_{n\in\mathbb{N}}$ be a process with homogeneous Markovian
dynamics and which takes values on
$\mathcal{C}=\{A\Subset\mathbb{Z}^d\}$. Let $C_0
\Subset\mathbb{Z}^d$ the set in which we want to observe the perfect
sampling from the Gibbs measure with infinite range interaction. If
$C_n = \emptyset $ then $C_{n+1 } = \emptyset $. If $C_n \neq
\emptyset$, then the set $C_{n+1 }$  is constructed as follows. A
random vertex $W_n$ is selected, proportionally to its mass, with
\begin{equation}\label{sortver}
\mathbb{P}(W_n =w|C_n)= \frac{M_w}{ \sum_{z \in C_n} M_z}, \,\,\,
\hbox{ for } w\in C_n.
\end{equation}
Formula (\ref{sortver}) will be used to define more general models
in Section~\ref{sec5}. Then a random value $K_{w,n}$ is drawn by
using the probability distribution $\lambda_w $, thus
$$
\mathbb{P}(K_{w,n} =k) =\lambda_w (k) ,  \hbox{ for } k \in
\mathbb{N}.
$$
 If
$ K_{w,n} =0 $ then $ C_{n+1 } = C_{n } \setminus \{ w\}$; if $
K_{w,n} =k $, for $ k\in \mathbb{N}_+$, then $ C_{n+1 } = C_{n }
\cup B_w(K_{w,n}) =C_{n } \cup B_w(k) $. The procedure ends at the
first time $m\in\mathbb{N}_+$ such that $C_m=\emptyset$. When this
happens, the forward spin procedure begins. Now the value of the
spin is assigned to all the vertices visited during the first stage,
starting at the last vertex with $k=0$. The assignment of spins is
done by using the update probabilities $p_v^{[k]}$, coming back up
to give the definitive value of the spin to the vertices belonging
to $C_0$.

\comment{ The following proposition is a summary of some results in
\cite{GLO} that we write in a more general form with respect to $
\mathbf{B}_v $. }

The following proposition characterizes the computability of the
algorithm and shows that there is an unique condition on the
backward sketch procedure and none on the second part of the
algorithm.
\begin{prop}\label{prop1}
 The perfect simulation algorithm in \cite{GLO} generates a random
field with distribution $\pi $ if and only if for any
$v\in\mathbb{Z}^d$
\begin{equation}\label{limsc}
    \limsup_{n \to \infty} C_n =\emptyset \,\,\,
    a.s.
\end{equation}
\end{prop}
\begin{proof}\label{proof1}
Condition (\ref{limsc}) is surely necessary  by definition of
\emph{algorithm}. It is also sufficient because it means that the
backward sketch procedure stops in a finite number of steps (almost
surely), moreover conditions (\ref{sommab}) and $ p_v^{[k]}(\cdot
|\sigma) \in [0,1]$, which hold by hypothesis and by construction
respectively, are sufficient for the forward spin procedure.
\end{proof}
 A sufficient condition, given in \cite{GLO}, for
(\ref{limsc}) is
\begin{itemize}
    \item[(H1)]
    $\,$
    \vspace{-8mm}
    \begin{equation*}\label{2mhp}
\sup_{v\in\mathbb{Z}^d}\sum_{k=1}^\infty|B^*_v(k)|\lambda_v(k)<1,
\end{equation*}
where $B^*_v(k)$ is the ball, in norm $L^1$, centered in $v$ with
radius $k$.
\end{itemize}

We provide a weaker sufficient condition for (\ref{limsc}) than (H1)
that is presented in the following theorem.
\begin{theorem} \label{affonda}
For a given  $ \mathbf{J} \in \mathcal{J}$,
\begin{itemize}
    \item[\emph{(H2)}]  if  a collection
$\{\mathbf{B}_v\in\mathcal{B}_v: v \in \mathbb{Z}^d \}$ such that
\begin{equation*}
    \lim_{\Lambda\uparrow \mathbb{Z}^d} \sup_{v \notin
\Lambda}\sum_{k=1}^\infty|B_v(k)|\lambda_v(k)<1
\end{equation*}
\end{itemize}
can be constructed,
 then (\ref{limsc}) holds. Hence {(H2)} is a sufficient
 condition for the perfect sampling from the Gibbs measure related to
 $\mathbf{J}$ (see Proposition \ref{prop1}).
\end{theorem}

In Section \ref{sec5} we will give the proof of this theorem.

\begin{rem}\label{AAA}
For a given $ \mathbf{J} \in \mathcal{J}$, if
$$
 \lim_{\Lambda\uparrow \mathbb{Z}^d} \sup_{v \notin
\Lambda}\min_{\mathbf{B}_v\in\mathcal{B}_v}
\sum_{k=1}^\infty|B_v(k)|\lambda_v(k)<1,
$$
then there exists a unique Gibbs measure verifying the local
specifications (see (\ref{spec})). Therefore {(H2)} can be seen also
as a sufficient condition for the uniqueness of the Gibbs measure.
In Theorem \ref{minimo}, we prove that the minimum in the previous
expression exists. Hence the  results on perfect simulation are
important also for the study of the transition phase, a classical
argument of the statistical mechanics.
 \end{rem}


\comment{ We will not be involved in the second step of the
algorithm for lack of space, but for convenience of the reader, we
will give the complete algorithm at the end of Section \ref{sec4}. }



\section{Stochastic ordering for $\lambda_{v, \mathbf{J}, \mathbf{B}_v}$ and an optimization problem for the perfect simulation }\label{sec3}

In this section we deal with the optimal choice of
$\mathbf{B}_v\in\mathcal{B}_v$, reaching concrete results. We start
with some definitions.

\begin{definition} \label{raffinata} For $v \in \mathbb{Z}^d$, the sequence $\mathbf{B}_v\in\mathcal{B}_v$ is
\emph{less refined}  than $\mathbf{B'}_v\in\mathcal{B}_v$, in
symbols $\mathbf{B}_v\preceq\mathbf{B'}_v$, if $\mathbf{B}_v$ is a
subsequence of $ \mathbf{B'}_v$.
\end{definition}

This relation between two sequences of  $\mathcal{B}_v$ is a partial
order. \comment{In fact the following properties are satisfied:
\\reflexivity: for each $\mathbf{B}\in\mathcal{B}$, $\mathbf{B}\preceq\mathbf{B}$;
\\antisymmetry: for each $\mathbf{B},\mathbf{C}\in\mathcal{B}$, if $\mathbf{B}\preceq\mathbf{C}$ and
$\mathbf{C}\preceq\mathbf{B}$, then $\mathbf{B}=\mathbf{C}$;
\\transitivity: for each $\mathbf{B_1},\mathbf{B_2},\mathbf{B_3}\in\mathcal{B}$,
if $\mathbf{B_1}\preceq\mathbf{B_2}$ and
$\mathbf{B_2}\preceq\mathbf{B_3}$, then
$\mathbf{B_1}\preceq\mathbf{B_3}$.

Hence the couple $(\mathcal{B},\preceq)$ is a partially ordered
set.} The set $\mathcal{B}_v$ has no minimum, nor maximum, nor even
minimal elements; nevertheless it has an uncountable infinite number
of maximal elements, corresponding to the sequences of sets which
increase by only one vertex at a time.


Let us define, for $v \in \mathbb{Z}^d$, a  probability distribution
obtained from $ \lambda_{v,\mathbf{J},\mathbf{B}_v}$ as follows
\begin{equation*}\label{dhat}
\begin{array}{lll}
&\hat{\lambda}_{v,\mathbf{J},\mathbf{B}_v}(|B_v(l)|-1)=\lambda_{v,\mathbf{J},\mathbf{B}_v}(l),
 &\hbox{for } l \in \mathbb{N},\\
&\hat{\lambda}_{v,\mathbf{J},\mathbf{B}_v}(i-1)=0, &\hbox{for }i\not\in\{|B_{v}(l)|,l\in\mathbb{N}\}.\\
\end{array}
\end{equation*}

\begin{theorem}\label{pr}
Let $v \in \mathbb{Z}^d$, $ \mathbf{J} \in \mathcal{J}$,  and
$\mathbf{B}_v$, $\mathbf{B'}_v\in\mathcal{B}_v$ such that
$\mathbf{B}_v\preceq\mathbf{B'}_v$.
 Then $\hat{\lambda}_{v,\mathbf{J},\mathbf{B'}_v}\preceq_{st}\hat{\lambda}_{v,\mathbf{J},\mathbf{B}_v}$.
\end{theorem}
\begin{proof}
For brevity of notation we write $ \hat{\lambda}_v=
\hat{\lambda}_{v,\mathbf{J},\mathbf{B}_v}$ and $ \hat{\lambda'}_v=
\hat{\lambda}_{v,\mathbf{J},\mathbf{B'}_v}$. To show the stochastic
ordering $\hat{\lambda'}_{v}\preceq_{st}\hat{\lambda}_{v}$ we
equivalently prove that for each  $n\in\mathbb{N}$,
\begin{equation}\label{l1}
F'(n)=\sum_{l=0}^n\hat{\lambda'}_v(l)\geq\sum_{l=0}^n\hat{\lambda}_v(l)
= F(n).
\end{equation}
The functions $F(n)$ and $F'(n)$ are the cumulative distribution
functions  relative to $\hat{\lambda}_v$  and $\hat{\lambda'}_v$
respectively. They are piecewise constant functions whose jumps
occur only in the points of the set $\{|B_v(l)|-1,l\in\mathbb{N}\}$
and $\{|B'_v(l)|-1,l\in\mathbb{N}\}$ respectively, i.e.
$$
F(n)=\sum_{l=0}^{n}\hat{\lambda}_v(l) =\sum_{l=0}^{j}{\lambda}_v(l)
, {\hbox{ where }}j=\max\{l\in\mathbb{N}:|B_v(l)|-1\leq n\},
$$
$$
F'(n)=\sum_{l=0}^{n}\hat{\lambda'}_v(l)
=\sum_{l=0}^{j'}{\lambda'}_v(l) , {\hbox{ where }}j'=\max\{l\in
\mathbb{N}:|B'_v(l)|-1\leq n\}.
$$
Now we show that for each $m\in\{|B_v(l)|-1,l\in\mathbb{N}\}$,
\begin{equation}\label{zanzara}
F(m)=F'(m).
\end{equation}
Let $m\in\{|B_v(l)|-1,l\in\mathbb{N}\}$, then
$$
F(m)=\sum_{l=0}^{j}{\lambda}_v(l), {\hbox{ where }}j{\hbox{ is the
unique index such that }}|B_v(j)|-1=m,
$$
$$
F'(m)=\sum_{l=0}^{j'}{\lambda'}_v(l), {\hbox{ where }}j'{\hbox{ is
the unique index such that }}|B'_v(j')|-1=m,
$$
from which, by the hypothesis of the theorem,
\begin{equation}\label{r!}
B_v(j)=B'_v(j').
\end{equation}

Note that the following sums are telescopic, hence
\begin{equation}\label{telescopio}
    \sum_{l=0}^n\lambda_v(l)=\exp\left(-\sum\nolimits_{B:v\in
B,B\not\subset B_v(n)}|J_B|\right) \hbox{ and }
\sum_{l=0}^n\lambda'_v(l)=\exp\left(-\sum\nolimits_{B:v\in
B,B\not\subset B'_v(n)}|J_B|\right),
\end{equation}
for $n\in\mathbb{N}_+$. Moreover $
F(0)=\hat{\lambda}_v(0)={\lambda}_v(0)={\lambda}'_v(0)=\hat{\lambda'}_v(0)=F'(0).
$

From (\ref{r!}) and (\ref{telescopio}),
$$
\sum_{l=0}^j\lambda_v(l)=\sum_{l=0}^{j'}\lambda'_v(l)
$$
immediately
follows and it implies (\ref{zanzara}). Since $F$ and $F'$ are
nondecreasing, from (\ref{zanzara}) and
$$
\{|B_v(l)|,l\in\mathbb{N}\}\subset\{|B'_v(l)|,l\in\mathbb{N}\}
$$
we obtain (\ref{l1}).
\end{proof}

 Analogously to \cite{GLO}, see
(H1), we introduce the following quantity that will be used later;
we call it \emph{birth-death expectation},
\begin{equation*}\label{media}
     \mu_{v,\mathbf{J}}(\mathbf{B}_v)=\sum_{l =1}^\infty
     |B_v(l)|\lambda_{v,\mathbf{J}, \mathbf{B}_v}(l)-1,
\end{equation*}
for $ \mathbf{J} \in \mathcal{J}$, $v \in \mathbb{Z}^d$,
$\mathbf{B}_v\in\mathcal{B}_v$.

We are now in the position to present our result concerning the
birth-death expectation, it will be involved in conditions (H1) and
(H2) for the perfect sampling.

\begin{corollary} \label{fonda}
Let  $ \mathbf{J} \in \mathcal{J}$, $v \in \mathbb{Z}^d$,
$\mathbf{B}_v$, $\mathbf{B'}_v\in\mathcal{B}_v$ such that
$\mathbf{B}_v\preceq\mathbf{B'}_v$. Then $\mu_{v,\mathbf{J}}(
\mathbf{B'}_v ) \leq \mu_{v,\mathbf{J}}( \mathbf{B}_v )$.
\end{corollary}
\begin{proof}
    Let  $ \mathbf{J}\in \mathcal{J}$, $v \in \mathbb{Z}^d$,
$\mathbf{B}_v$, $\mathbf{B'}_v\in\mathcal{B}_v$ such that
$\mathbf{B}_v\preceq\mathbf{B'}_v$ and let
$\hat{\lambda}_v=\hat{\lambda}_{v,\mathbf{J},\mathbf{B}_v}$,
$\hat{\lambda'}_v=\hat{\lambda}_{v,\mathbf{J},\mathbf{B'}_v}$ be the
corresponding measures. Consider two random variables $ X_v
\sim^\mathcal{L}\hat{\lambda}_v $ and $ X'_v
\sim^\mathcal{L}\hat{\lambda'}_v $. From Theorem \ref{pr}, it
follows that $ \mathbb{E}( f (X_v)) \geq \mathbb{E}( f (X'_v))$ for
each nondecreasing function $ f: \mathbb{N}\to \mathbb{R} $. Note
that
\begin{equation}\label{mediano}
     \mu_{v,\mathbf{J}}(\mathbf{B}_v)=\sum_{l =1}^\infty |B_v(l)|\lambda_v(l)-1
     = \sum_{l =1}^\infty (|B_v(l)|-1)\lambda_v(l)-\lambda_v(0)
\end{equation}
\begin{equation*}
 =
   \sum_{l =1}^\infty
   (|B_v(l)|-1)\hat{\lambda}_v(|B_v(l)|-1)-\hat{\lambda}_v(0) =
   \sum_{i =1}^\infty (i-1)\hat{\lambda}_v(i-1) -\hat{\lambda}_v(0)
   =\sum_{i =1}^\infty i\ \hat{\lambda}_v(i) -\hat{\lambda}_v(0),
\end{equation*}
therefore (\ref{mediano}) is the expected value of the random
variable $g( X_v)$ where,
\begin{equation}\label{funz}
g(i) =\left \{ \begin{array}{cc}
         -1 &\hbox{ if } i=0 ,\\
                  i &\hbox{ if } i\geq 1. \\
                \end{array} \right .
\end{equation}
The function in (\ref{funz}) is nondecreasing. Thus, by the
stochastic ordering, $\mu_{v,\mathbf{J}}( \mathbf{B'}_v
)=\mathbb{E}( g (X'_v)) \leq \mathbb{E}( g (X_v))=\mu_{v,\mathbf{J}}
( \mathbf{B}_v )$.
\end{proof}
\comment{ ==========================================
Finally we show
the usefulness of our choice of $\lambda_v$ which is simpler than
theirs.

\begin{rem}\label{ossla}
 For each $v\in\mathbb{Z}^d$ the sequence
$(\lambda_v(k))_{k\geq0}$, defined in (\ref{lambda}), is preferable
than the one given in \cite{GLO} which we indicate with
$(\lambda^*_v(k))_{k\geq0}$. Indeed, given an interaction
$\mathbf{J}= \{J_B\in\mathbb{R},B\Subset\mathbb{Z}^d\}$, since
$\lambda_v(0)>\lambda_v^*(0)$, $\lambda_v(1)<\lambda_v^*(1)$ and
$\lambda_v(k)=\lambda_v^*(k)$ for $k\geq2$, then the two measures
respect the stochastic ordering $\lambda_v\preceq_{st}\lambda_v^*$
for each $v\in\mathbb{Z}^d$. Since
$\hat{\lambda}_v(0)=\lambda_v(0)>\lambda^*_v(0)=\hat{\lambda}^*_v(0)$
and $\hat{\lambda}_v(0)+\hat{\lambda}_v(k)=
\hat{\lambda}^*_v(0)+\hat{\lambda}^*_v(k)$ for $k\geq|B_v(1)|-1$,
then $\hat{\lambda}_v\preceq_{st}\hat{\lambda}_v^*$. From the latter
stochastic ordering, the equalities in (\ref{mediano}), and the fact
that the function in (\ref{funz}) is nondecreasing, (\ref{media})
calculated by $\lambda_v$ is smaller than (\ref{media}) calculated
by $\lambda^*_v$. Hence our choice of $\lambda_v$ facilitates the
conditions (H1) and  (H2) for the applicability of the algorithm.
\end{rem}
===================================}

By the next two theorems, we will see that if an interaction
$\mathbf{J}$ verifies (H1), then all the interactions obtained from
$\mathbf{J}$ by changing them on a finite region and by lowing them
in absolute value elsewhere, still verify (H2).  By Theorem
\ref{affonda}, all the  Gibbs measures associated to these
interactions are perfectly simulable.

\begin{theorem}\label{pr2}
Let $v \in \mathbb{Z}^d$, $\mathbf{B}_v\in\mathcal{B}_v$,
$\mathbf{J}$, $\mathbf{\widetilde{J}} \in \mathcal{J}$ such that
$|\widetilde{J}_B|\leq|J_B|$ for each $B\Subset\mathbb{Z}^d$. Then
$\lambda_{v,\mathbf{\widetilde{J}},\mathbf{B}_v}
\preceq_{st}\lambda_{v,\mathbf{J},\mathbf{B}_v}$. Hence
$\mu_{v,\mathbf{\widetilde{J}}}(\mathbf{B}_v) \leq
\mu_{v,\mathbf{J}}(\mathbf{B}_v)$.
\end{theorem}
\begin{proof}
For brevity of notation we write $ \lambda_v=
\lambda_{v,\mathbf{J},\mathbf{B}_v}$ and $ \widetilde{\lambda}_v=
\lambda_{v,\mathbf{\widetilde{J}},\mathbf{B}_v}$. To show the
stochastic ordering, we equivalently prove that for each
$v\in\mathbb{Z}^d$, $n\in\mathbb{N}$
\begin{equation*}
\sum_{l=0}^n\widetilde{\lambda}_v(l)\geq\sum_{l=0}^n{\lambda}_v(l).
\end{equation*}
Since $|\widetilde{J}_B|\leq|J_B|$ for each $B\Subset\mathbb{Z}^d$,
then
$$
\widetilde{\lambda}_v(0)=\exp\left(-2\sum\nolimits_{B:v\in
B}|\widetilde{J}_B|\right) \geq\exp\left(-2\sum\nolimits_{B:v\in
B}|J_B|\right)=\lambda_v(0),
$$
and for $n\geq1$
$$
\sum_{l=0}^n\widetilde{\lambda}_v(l)=\exp\left(-\sum\nolimits_{B:v\in
B,B\not\subset B_v(n)}|\widetilde{J}_B|\right)\geq
\exp\left(-\sum\nolimits_{B:v\in B,B\not\subset B_v(n)}|J_B|\right)
=\sum_{l=0}^n\lambda_v(l).
$$
\end{proof}

The following result is directly related to our sufficient condition
(H2).
\begin{theorem}\label{pr3}
Given the interactions $\mathbf{J}$,
$\mathbf{\widetilde{J}}\in\mathcal{J}$, if the cardinality of
$\mathcal{C}=\{ B\Subset\mathbb{Z}^d: |J_B| \neq |\widetilde{J}_B|
\}$ is finite, then for $v\in\mathbb{Z}^d$ and
$\mathbf{B}_v\in\mathcal{B}_v$, \comment{the corresponding measures
$\lambda_v=\lambda_{v,\mathbf{J},\mathbf{B}_v}$,
$\widetilde{\lambda}_v=\lambda_{v,\mathbf{\widetilde{J}},\mathbf{B}_v}$
are such that}
\begin{equation}\label{emi}
\lim_{\Lambda\uparrow \mathbb{Z}^d} \sup_{v \notin
\Lambda}\mu_{v,\mathbf{J}}(\mathbf{B}_v)= \lim_{\Lambda\uparrow
\mathbb{Z}^d} \sup_{v \notin
\Lambda}\mu_{v,\mathbf{\widetilde{J}}}(\mathbf{B}_v).
\end{equation}
\end{theorem}
\begin{proof}\label{prova3}
Note that the measures $\lambda_{v,\mathbf{J},\mathbf{B}_v}$,
$\lambda_{v,\mathbf{\widetilde{J}},\mathbf{B}_v}$ are equal for each
$v$ such that all the finite subsets $B$ containing $ v$ do not
belong to $\mathcal{C}$. In fact if $\{B\Subset\mathbb{Z}^d:v\in
B,B\in\mathcal{C}\}=\emptyset$, then for each $B$ including  $v$ we
have $|J_B|=|\widetilde{J}_B|$, hence
$\lambda_{v,\mathbf{J},\mathbf{B}_v}=
\lambda_{v,\mathbf{\widetilde{J}},\mathbf{B}_v}$ for each $k\geq0$.
Therefore for $ \Lambda \supset \bigcup_{B \in \mathcal{C}} B  $,
\begin{equation}\label{passando}
 \sup_{v \notin \Lambda}\mu_{v,\mathbf{J}}(\mathbf{B}_v)+1=
 \sup_{v \notin \Lambda}\sum_{k=1}^\infty|B_v(k)|\lambda_{v,\mathbf{J},\mathbf{B}_v}(k)=
 \sup_{v \notin
\Lambda}\sum_{k=1}^\infty|B_v(k)|\lambda_{v,\mathbf{\widetilde{J}},\mathbf{B}_v}(k)=
\sup_{v \notin
\Lambda}\mu_{v,\mathbf{\widetilde{J}}}(\mathbf{B}_v)+1.
\end{equation}
 Since the
cardinality of $\mathcal{C}$ is finite, then $\bigcup_{B \in
\mathcal{C}}B$ is finite. Therefore, passing to the limit in
(\ref{passando}) for $\Lambda\uparrow\mathbb{Z}^d$, we obtain
(\ref{emi}).
\end{proof}

Condition (H2) says that $  \lim_{\Lambda\uparrow \mathbb{Z}^d}
\sup_{v \notin \Lambda} \mu_{v,\mathbf{J}}(\mathbf{B}_v) <0 $,
therefore we are interested in finding the infimum value $   \inf_{
\mathbf{x} \in \mathcal{B}_v } \mu_{v,\mathbf{J}}( \mathbf{x})  $.

 We define $ \mathcal{E}_v $ by distinguishing two cases  $ N_v =
 \infty$,  $ N_v <\infty$. In the first case  let $ \mathcal{E}_v $ be a
  subset of $
\mathcal{B}_v$ such that each  element $ (B_v (l))_{l \in
\mathbb{N}} \in \mathcal{E}_v $ has the property that there exists a
sequence $(i_k)_{k\in\mathbb{N}}$ where $ B_v (l)= \bigcup_{k=1}^l
A_{i_k, v} $ for $ l \in \mathbb{N}_+ $. When $ N_v <\infty$, let $
\mathcal{E}_v $ be a
  subset of $
\mathcal{B}_v$ such that each element $ (B_v (l))_{l \in \mathbb{N}}
\in \mathcal{E}_v $ has the property that
\begin{equation}\label{cena}
\exists\bar{l}: B_v(\bar{l})=\bigcup_{k=1}^{N_v}A_{k,v},\,\, \exists
(i_1 , \ldots , i_{\bar{l}}) : B_v(l)=\bigcup_{k=1}^{l}A_{i_k,v}\
\forall l\leq\bar{l}.
\end{equation}
We notice that, for each $l>\bar{l}$, $\lambda_v(l)=0$ for any
choice of $B_v(l)$ verifying 2).

In the next theorem we restrict the research of the infimum from
$\mathcal{B}_v $ to $ \mathcal{E}_v$. This produces a sensitive
improvement when $N_{v}$ is finite for each vertex $ v \in
\mathbb{Z}^d$ , in this case the infimum is a minimum because there
is a finite number of choices in (\ref{cena}), and this fact allows us to calculate it.
In any case, in Theorem \ref{minimo} we will prove that  the minimum of $
\mu_{v,\mathbf{J}}( \mathbf{x})$ always exists.

We endow $ \mathcal{B}_v$ with the discrete topology to consider the
limit of a sequence in $\mathcal{B}_v $ in the next two theorems.

\begin{theorem}\label{Thscelta}
Let $\mathbf{J}\in\mathcal{J}$, $v \in \mathbb{Z}^d$, then
  $$
  \inf_{ \mathbf{x} \in \mathcal{B}_v } \mu_{v,\mathbf{J}}( \mathbf{x}) =
  \inf_{ \mathbf{x} \in \mathcal{E}_v } \mu_{v,\mathbf{J}}(
  \mathbf{x}).
  $$
\end{theorem}
\begin{proof}
 First
we consider the case $N_v = \infty$.  To prove the theorem we will
show that for each $\mathbf{x} \in \mathcal{B}_v$ there exists $
\mathbf{y} \in \mathcal{E}_v$ such that $ \mu_{v, \mathbf{J}}
(\mathbf{y}) \leq \mu_{v, \mathbf{J}}(\mathbf{x}) $. Starting
from $ \mathbf{x} = (x(l))_{l \in \mathbb{N}} \in \mathcal{B}_v $,
we will construct a sequence of points $( \mathbf{x}^{(n)} \in
\mathcal{B}_v )_{n \in \mathbb{N}}$ such that $ \mathbf{x}^{(0)}
=\mathbf{x}$ and $ \lim_{n \to \infty}\mathbf{x}^{(n)} = \mathbf{y}
\in \mathcal{E}_v $. We will prove that, for $n \in \mathbb{N}$, $
\mu_{v, \mathbf{J}} ( \mathbf{x}^{(n+1) }) \leq  \mu_{v, \mathbf{J}}
( \mathbf{x}^{(n) })$ and then, by Fatou's lemma, $ \mu_{v,
\mathbf{J}}(\mathbf{y} ) \leq \liminf_{n \to \infty} \mu_{v,
\mathbf{J}} ( \mathbf{x}^{(n) })$, from which $ \mu_{v, \mathbf{J}}
(\mathbf{y}) \leq \mu_{v, \mathbf{J}} (\mathbf{x})$.

Let $\mathbf{x}^{(0)} = \mathbf{x}=(x(l))_{l \in \mathbb{N}} \in
\mathcal{B}_v $, we now give the rules to construct
$\mathbf{x}^{(1)}$. Define
$$
k_0=1+ \sup\{ l\in\mathbb{N}_+ : \exists(i_1, \ldots ,i_l)\hbox{
s.t. }x(j)=\bigcup_{k=1}^j A_{i_k, v} \hbox{ for any } j=1, \ldots,
l \},
$$
if $ k_0 = \infty$, then $ \mathbf{x}\in \mathcal{E}_v $ and there
is nothing to prove. If $  k_0 < \infty$ then define the finite sets
of indices
\begin{equation*}\label{I}
    I = \{ i \in \mathbb{N}_+: A_{i,v} \subset x(k_0 ) \} ,
\end{equation*}
\begin{equation*}\label{I-}
      I^- = \{ i \in \mathbb{N}_+: A_{i,v} \subset x(k_0 -1) \} .
\end{equation*}
If $ I= I^-  $ then eliminate $ x(k_0 ) $ from the sequence
obtaining $ x^{(1)}(l) = x(l ) $, for $ l \leq  k_0 -1$, $
x^{(1)}(l) = x(l +1) $, for $ l \geq  k_0 $. In this case $ \mu_{ v,
\mathbf{J}} (\mathbf{x}^{(0)})  =\mu_{ v, \mathbf{J}}
(\mathbf{x}^{(1)}) $.

If $  I\neq I^- $, consider  $ j = \min \{i : i \in  I \setminus I^-
\}$, define $ x^{(1)}(l) = x (l ) $, for $ l \leq k_0 -1$, $
x^{(1)}(k_0) = x (k_0-1 ) \cup A_{j , v  } $,  $ x^{(1)}(l) = x(l
-1) $, for $ l \geq  k_0+1 $. It is easy to check that the sequence
$ \mathbf{x}^{(1)}  $ verify the conditions 1), 2) and 3) defining
$\mathcal{B}_v$. In this case the sequence $\mathbf{x}^{(0)}$ is
less refined than $ \mathbf{x}^{(1)}$, therefore $ \mu_{v,
\mathbf{J}}(\mathbf{x}^{(0)} )  \geq \mu_{v,
\mathbf{J}}(\mathbf{x}^{(1)} )$, by Corollary \ref{fonda}.

We repeat the procedure to construct $ \mathbf{x}^{(n+1)} $ from $
\mathbf{x}^{(n)}$, for any $ n \in \mathbb{N_+}$. Obviously there
exists $ \lim_{n \to \infty } \mathbf{x}^{(n)}= \mathbf{y} \in
\mathcal{E}_v$.
Since $\hat{\lambda}_{ v,\mathbf{J}, \mathbf{z}}(0 )$ does not
depend on $ \mathbf{z} \in \mathcal{B}_v$ we set $\hat{\lambda}_{
v,\mathbf{J}}(0 )=\hat{\lambda}_{ v,\mathbf{J}, \mathbf{z}}(0 ) $,
therefore we can write
$$
\mu_{ v, \mathbf{J}}(\mathbf{y}) =-\hat{\lambda}_{v, \mathbf{J},
\mathbf{y}} (0) +\sum_{ i=1 }^{\infty} i \hat{\lambda}_{v,
\mathbf{J}, \mathbf{y}} (i) = -\hat{\lambda}_{ v, \mathbf{J}}(0)
+\sum_{ i=1 }^{\infty} \liminf_{n \to \infty}  i \hat{\lambda}_{
v,\mathbf{J}, \mathbf{x}^{(n)}} (i)
$$
$$ \leq
-\hat{\lambda}_{ v, \mathbf{J}}(0) + \liminf_{n \to \infty}  \sum_{
i=1 }^{\infty} i \hat{\lambda}_{ v,\mathbf{J}, \mathbf{x}^{(n)}} (i)
= \liminf_{n \to \infty} \mu_{ v, \mathbf{J}}(\mathbf{x}^{(n)})\leq
\mu_{ v, \mathbf{J}}(\mathbf{x}) ,
$$
where the first inequality follows by Fatou's lemma. The case $ N_v
< \infty $ is simpler and in a finite number $n_0$ of steps one
obtains that $ \mathbf{x}^{(n_0)}\in\mathcal{E}_v$.
\end{proof}

Now we state the theorem on the minimum that has a theoretical
flavor but we will see that in some important cases the point
realizing the minimum  can be explicitly calculated.

\begin{theorem}\label{minimo}
Let $\mathbf{J}\in\mathcal{J}$, $v \in \mathbb{Z}^d$, then
  $$
  \min_{ \mathbf{x} \in \mathcal{B}_v } \mu_{v,\mathbf{J}}( \mathbf{x}) =
  \min_{ \mathbf{x} \in \mathcal{E}_v } \mu_{v,\mathbf{J}}(
  \mathbf{x}).
  $$
\end{theorem}
\begin{proof}
The theorem is obviously true in the case of $N_v< \infty $, thus we
 consider  $N_v = \infty$. 
 First we prove the existence of $ \min_{ \mathbf{x} \in \mathcal{B}_v } \mu_{v,\mathbf{J}}(
 \mathbf{x})$. If for each $ \mathbf{x} \in \mathcal{B}_v $,  $ \mu_{v,\mathbf{J}}(
\mathbf{x}) =\infty$, there is nothing to show. Suppose that for $
\mathbf{x}= (x(l))_{l \in \mathbb{N}} $,
$$ \mu_{v,\mathbf{J}}( \mathbf{x}) \leq c <\infty.
$$
Define, for any $A_{k , v}  \in \mathcal{A}_v$,
\begin{equation*}\label{indicel}
    \bar{l}(k) = \min \{ l \in \mathbb{N}: x(l) \supset A_{k , v}
    \},
\end{equation*}
therefore
\begin{equation}\label{facileserve}
 A_{k , v} \not \subset  x( \bar{l}(k) -1 ), \,\,\,  A_{k , v}
 \subset  x( \bar{l}(k)  ).
\end{equation}
We will prove that for each $ k \in  \mathbb{N}_+ $
\begin{equation}\label{funzfin}
\bar{l}(k) \leq 
\frac{c+1}{ e^{-L+ |J_{A_{k,v}}|} - e^{-L}} \vee 2,
\end{equation}
where $ L = \sum_{ B: v \in B } |J_B|$.

 Since $ \mu_{v,\mathbf{J}}( \mathbf{x}) \leq c$, then
 \begin{equation}\label{unsoloterm}
( \bar{l}(k)+1) \lambda_{v , \mathbf{J}, \mathbf{x}} ( \bar{l}(k)
)\leq \sum_{l=1}^{\infty}|x(l )|   \lambda_{v , \mathbf{J},
\mathbf{x}}(l)=\mu_{v,\mathbf{J}}( \mathbf{x}) +1\leq c+1,
\end{equation}
where the first inequality is true because we have only taken a term
of the sum and used that $ |x(l )| \geq l+1 $.

 \comment{ By
contradiction suppose that
\begin{equation}\label{assurdaf}
\bar{l}(k)>\frac{c+1}{ e^{-L+ |J_{A_{k,v}}|} - e^{-L}} \vee 2  \geq
2 .
\end{equation}
}

Let $ S_k=\sum_{B :v \in B, B \not \subset x (k-1)}  |J_B|$. Since $
S_k \in [0, L]$, then
\begin{equation}\label{piuche}
    e^{-L+ |J_{A_{k,v}}|} - e^{-L} \leq e^{-S_k+ |J_{A_{k,v}}|} -
e^{-S_k}\leq \lambda_{v , \mathbf{J}, \mathbf{x}} ( \bar{l}(k) ),
\end{equation}
where the last inequality follows from (\ref{facileserve}) and from
the expression of $\lambda_{v , \mathbf{J}, \mathbf{x}} (l)$ (see
(\ref{lambda})) for $l \geq 2$. By (\ref{unsoloterm}) and
(\ref{piuche}) we obtain
\begin{equation*}\label{weare}
( \bar{l}(k)+1)  (e^{-L+ |J_{A_{k,v}}|} - e^{-L})\leq c+1,
\end{equation*}
which implies  (\ref{funzfin}).

Now define a sequence $ (\mathbf{x}^{(n)} \in \mathcal{B}_v: n \in
\mathbb{N}   )$, such that $  \mathbf{x}^{(0)} = \mathbf{x}$, the
birth-death expectations $ \mu_{v,\mathbf{J}}( \mathbf{x}^{(n)}) $
are nonincreasing in $n $ and $ \lim_{n \to \infty}
\mu_{v,\mathbf{J}}( \mathbf{x}^{(n)}) = \inf_{ \mathbf{x } \in
\mathcal{B}_v } \mu_{v,\mathbf{J}}( \mathbf{x})$. Let $
\bar{l}^{(n)}(k)= \min \{ l \in \mathbb{N}: x^{(n)}(l) \supset A_{k
, v} \} $ the analogous of $ \bar{l}(k)$.
By (\ref{funzfin}) there exists a subsequence $ ( \mathbf{x}_1^{(n)}
\in \mathcal{B}_v: n \in \mathbb{N}   )$ of $ (\mathbf{x}^{(n)} \in
\mathcal{B}_v: n \in \mathbb{N}   ) $ such that $ \bar{l}^{(n)}(1) $
is constant in $n$. For each $ h \in \mathbb{N}_+$ there exists a
subsequence $ ( \mathbf{x}_h^{(n)} \in \mathcal{B}_v: n \in
\mathbb{N}   )$ of $   (\mathbf{x}_{h-1}^{(n)} \in \mathcal{B}_v: n
\in \mathbb{N}   ) $  such that $ \bar{l}^{(n)}(h) $ is constant in
$n$. Therefore, by using diagonal method, the sequence $ (
\mathbf{x}_n^{(n)} \in \mathcal{B}_v: n \in \mathbb{N}   )$ admits
limit, i.e.
\begin{equation*}\label{dovemin}
    \lim_{n \to \infty} \mathbf{x}_n^{(n)} = \mathbf{y } \in
    \mathcal{B}_v.
\end{equation*}
 The sequence $ ( \mathbf{x}_n^{(n)} \in
\mathcal{B}_v: n \in \mathbb{N} )$ is a subsequence of the initial
one $ (\mathbf{x}^{(n)} \in \mathcal{B}_v: n \in \mathbb{N} ) $;
hence $ \mu_{v,\mathbf{J}}( \mathbf{x}_n^{(n)}) $ is nonincreasing
in $n $ and $ \lim_{n \to \infty} \mu_{v,\mathbf{J}}(
\mathbf{x}_n^{(n)}) = \inf_{ \mathbf{x } \in \mathcal{B}_v }
\mu_{v,\mathbf{J}}( \mathbf{x})$. Now, by Fatou's Lemma,
$$
\mu_{ v, \mathbf{J}}(\mathbf{y}) =-\hat{\lambda}_{v, \mathbf{J},
\mathbf{y}} (0) +\sum_{ i=1 }^{\infty} i \hat{\lambda}_{v,
\mathbf{J}, \mathbf{y}} (i) = -\hat{\lambda}_{ v, \mathbf{J}}(0)
+\sum_{ i=1 }^{\infty} \liminf_{n \to \infty}  i \hat{\lambda}_{
v,\mathbf{J}, \mathbf{x}_n^{(n)}} (i)
$$
$$ \leq
-\hat{\lambda}_{ v, \mathbf{J}}(0) + \liminf_{n \to \infty}  \sum_{
i=1 }^{\infty} i \hat{\lambda}_{ v,\mathbf{J}, \mathbf{x}_n^{(n)}}
(i) = \liminf_{n \to \infty} \mu_{ v,
\mathbf{J}}(\mathbf{x}_n^{(n)}) = \inf_{ \mathbf{x } \in
\mathcal{B}_v } \mu_{v,\mathbf{J}}( \mathbf{x}) .
$$
Therefore
\begin{equation}\label{molo17}
\mu_{v,\mathbf{J}}( \mathbf{y})= \inf_{ \mathbf{x } \in
\mathcal{B}_v } \mu_{v,\mathbf{J}}( \mathbf{x}),
\end{equation}
  that implies the
existence of the minimum on $ \mathcal{B}_v$.

For $ \mathbf{y} \in
\mathcal{B}_v  $ as in (\ref{molo17}), there exists $ \mathbf{z} \in
\mathcal{E}_v $ such that $  \mu_{v,\mathbf{J}}( \mathbf{y}) \geq
\mu_{v,\mathbf{J}}( \mathbf{z})$ (see the proof of Theorem
\ref{Thscelta}). It is immediately   seen that $ \mu_{v,\mathbf{J}}(
\mathbf{z}) = \inf_{ \mathbf{x } \in \mathcal{E}_v }
\mu_{v,\mathbf{J}}( \mathbf{x})$.
\end{proof}

In some cases it is possible to identify the sequence $ \mathbf{z}_v
\in \mathcal{B}_v$ such that $  \mu_{v,\mathbf{J}}( \mathbf{z}_v) =
\min_{ \mathbf{x } \in \mathcal{B}_v } \mu_{v,\mathbf{J}}(
\mathbf{x}) $.  We present a  result on  the Ising models in which
it occurs. Let $ \mathcal{J}_2 \subset   \mathcal{J}$ be the set of
the interactions such that
\begin{equation*}\label{j2}
    J_B \neq 0 \Rightarrow |B | =2 .
\end{equation*}
Note that under condition (\ref{sommab}) one gets
\begin{equation*}
    \lim_{\Lambda \uparrow \mathbb{Z}^d } \sup_{ u \not \in \Lambda} J_{\{
    v,u\}} =0\hbox{ for all }v \in \mathbb{Z}^d,
\end{equation*}
therefore it can be written, for a fixed vertex $v$,
\begin{equation}\label{ordinevert}
    | J_{\{ v, w_1 \}}| \geq  | J_{\{ v, w_2 \}}| \geq \ldots \geq  | J_{\{ v, w_i
    \}}| \geq
    \ldots
\end{equation}
where $ \bigcup_{i=1}^{\infty} w_i =\mathbb{Z}^d$ and $ w_i \neq w_j
$ if $ i \neq j$.
 Let  $ ( w_n \in \mathbb{Z}^d : n \in
\mathbb{N}_+)$ be the sequence written in (\ref{ordinevert}) and
define  $ \mathbf{z}_v \in \mathcal{B}_v $ such that, for $ i \in
\mathbb{N}_+$,
\begin{equation*}\label{aqoll}
    z_v ( i ) =\{ v, w_1 , \ldots , w_{i}\} .
\end{equation*}
Let also $ J^{(n)}_v =| J_{\{ v, w_n \}}| $ for $ n\in \mathbb{N}_+
$. We remark that, given an interaction $\mathbf{J } \in
\mathcal{J}_2$, the sequence $ ( w_n \in \mathbb{Z}^d : n \in
\mathbb{N}_+)$ is not in general unique.
\begin{theorem}\label{ThIsing}
    Let $ \mathbf{J} \in   \mathcal{J}_2$, for each $ v \in
    \mathbb{Z}^d$,
     \begin{equation}\label{Isingmin}
 \mu_{v,\mathbf{J}}( \mathbf{z}_v) =
\min_{ \mathbf{x } \in \mathcal{B}_v } \mu_{v,\mathbf{J}}(
\mathbf{x}) =-2e^{-2\sum_{i=1}^\infty
J^{(i)}_v}+e^{-\sum_{i=2}^\infty J^{(i)}_v}+\sum_{l=2}^\infty
l\left(e^{-\sum_{i=l+1}^\infty J^{(i)}_v}-e^{-\sum_{i=l}^\infty
J^{(i)}_v}\right).
\end{equation}
\end{theorem}
\begin{proof}
First notice that $  \mathbf{z}_v $ is a maximal element of $
\mathcal{B}_v $ and $ \mathbf{z}_v  \in  \mathcal{E}_v$, moreover
$\lambda_{ v, \mathbf{J}, \mathbf{z}_v} = \hat{\lambda}_{ v,
\mathbf{J}, \mathbf{z}_v}$. To prove the theorem we will show that,
for each maximal element $\mathbf{x } \in \mathcal{B}_v  $, we
obtain $ \lambda_{ v, \mathbf{J}, \mathbf{z}_v} \preceq_{st}
\lambda_{ v, \mathbf{J}, \mathbf{x}}$. Hence by Theorem \ref{pr} we
will get the first equality of(\ref{Isingmin}).

Let $ \mathbf{x}$ be a maximal element of $ \mathcal{B}_v $, as in
the proof of Theorem \ref{pr}, we show that for any $ n \in
\mathbb{N}$
\begin{equation}\label{tardi}
\sum_{l=0}^n  \lambda_{ v, \mathbf{J},
\mathbf{z}_v}(l)\geq\sum_{l=0}^n {\lambda}_{ v, \mathbf{J},
\mathbf{x}}(l) ,
\end{equation}
which guarantees the stochastic ordering. The l.h.s. of
(\ref{tardi}) is
$$
\exp{ \left (-\sum \nolimits_{ B: v \in B, B \not \subset z_v
(n)}|J_B| \right ) } =\exp{ \left (-\sum_{l=n+1}^{\infty}|J_{\{v,
w_l \}}| \right ) }.
$$
Consider the sequence of  distinct vertices $ \{ u_n \in
\mathbb{Z}^d : n \in \mathbb{N}_+\}$ such that $u_0 = \{v\}$ and $
u_n = x(n) \setminus x(n-1)$. Since the r.h.s. of (\ref{tardi}) can
be written
$$
\exp{ \left (-\sum\nolimits_{B: v \in B, B \not \subset x (n)}|J_B|
\right ) } =\exp{ \left (-\sum_{l=n+1}^{\infty}|J_{\{v, u_l \}}|
\right ) },
$$
then inequality (\ref{tardi}) is equivalent to
$$
\exp{ \left (-\sum_{l=n+1}^{\infty}|J_{\{v, w_l \}}| \right ) }\geq
\exp{ \left (-\sum_{l=n+1}^{\infty}|J_{\{v, u_l \}}| \right ) }
$$
or
$$
\exp{ \left (\sum_{l=1}^{n}|J_{\{v, w_l \}}| \right ) }\geq \exp{
\left (\sum_{l=1}^{n}|J_{\{v, u_l \}}| \right ) }
$$
that is obviously true by using the definition of sequence $ \{ w_n\}$
(see (\ref{ordinevert})).

The second equality in (\ref{Isingmin}) follows by elementary
calculations.
\end{proof}

\begin{rem}\label{esplicito}
If for any $ v \in \mathbb{Z}^d$ the number $ N_v $ is \emph{small}
and if it can be proved that for some
$(\mathbf{x}_v\in\mathcal{B}_v)_{v\in\mathbb{Z}^d}$
\begin{equation}\label{luglio}
\lim_{\Lambda\uparrow\mathbb{Z}^d} \sup_{v \notin \Lambda} \mu_{v,
\mathbf{J}} (\mathbf{x}_{v} )<0,
\end{equation}
then the perfect simulation algorithm can be run. Proving
(\ref{luglio}) is a little easier than proving condition (H1), and
in both cases it should be done \emph{a priori}. In the backward
sketch procedure a random vertex $w$ is selected with probability
(\ref{sortver}), now the algorithm calculates all the $
\widehat{\mathbf{x}}_w$'s belonging to $\emph{arg}\min_{ \mathbf{x}
\in \mathcal{E}_w } \mu_{w, \mathbf{J}} ( \mathbf{x}) $ with a
finite number of elementary operations because, for any
$\mathbf{x}\in \mathcal{E}_w $, $\lambda_{w,\mathbf{J}, \mathbf{x}
}(l)$ must be calculated for $l = 1, \ldots , N_w$ and also all the
sums involved in the definition of $ \lambda_{w,\mathbf{J},
\mathbf{x} } $ and of $ \mu_{w,\mathbf{J} } (\mathbf{x})$ are
finite.
Moreover $|\mathcal{E}_w|\leq N_w!$.
By comparing the finite list (having
at most $N_w!$ elements) of $ \mu_{w,\mathbf{J} } (\mathbf{x})$ with
$ \mathbf{x}\in \mathcal{E}_w$,  the algorithm finds all the
 $\widehat{\mathbf{x}}_{w}$'s belonging to $  \mathcal{E}_w$ such that $ \mu_{w,
\mathbf{J}} ( \widehat{\mathbf{x}}_{w} ) =  \min_{ \mathbf{x} \in
\mathcal{E}_w} \mu_{w, \mathbf{J}} ( \mathbf{x})  $. This procedure
is repeated for all the selected vertices, which are almost surely
finite. Hence the problem is computable and the previous procedure
is really an algorithm. The computability is guaranteed by the fact
that $N_v$ is finite, further the algorithm runs in reasonable time
if $N_v$ is small.

If $N_v$ is \emph{large} or equal to \emph{infinity}, if one
succeeds in calculating a
$(\mathbf{x}_v\in\mathcal{B}_v)_{v\in\mathbb{Z}^d}$ such that
condition (\ref{luglio}) is satisfied, then the algorithm can use
this particular choice.

The case $ N_v = \infty $ is in some sense theoretical but there are
models in which a change of the first terms of a given  sequence $
\mathbf{B}_v $ may produce a \emph{sensitive improvement} for $
\mu_{v , \mathbf{J}}( \mathbf{B}_v ) $, i.e. it goes from positive
values to negative values.
\comment{ that can be dealt under general
conditions. We now present a class of infinite-range models in
which, using the $\varepsilon$-optimality, we can find, for a given
$ \varepsilon
>0$, a sequence $ \mathbf{B}_v $ such that $ \mu_{v, \mathbf{J}} (
\mathbf{B}_v) - \inf_{ \mathbf{x} \in \mathcal{E}_v }
\mu_{v,\mathbf{J}}( \mathbf{x}) \leq \varepsilon $. } For simplicity
of the exposition we only consider  translation invariant models.
 Let us assume that, for a fixed $ \mathbf{\widehat{B}}_v \in\mathcal{B}_v $,
 \begin{equation}\label{sommafin}
    \sum_{k=1}^{\infty} | \widehat{B}_v( k) | \lambda_{ v , \mathbf{J},
    \mathbf{\widehat{B}}_v} ( k) < \infty .
\end{equation}
Hence, for $ N\in \mathbb{N}$, we
consider the finite subset $ \Upsilon_{N} (\mathbf{\widehat{B}}_v )$
of $ \mathcal{B}_v$ made by all the sequences verifying these rules:
\begin{itemize}
    \item  $ \widetilde{B}_v (0)
    \subset \widetilde{B}_v (1)  \subset \ldots \subset  \widetilde{B}_v (L) =\widehat{B}_v (N)
    $, with $L = | \widehat{B}_v (N) | -1$;
   \item  $ \widetilde{B}_v (L+i) = \widehat{B}_v (N+i)$, for $ i
\geq 1$.
\end{itemize}
It is easy to calculate
$$
\min_{ \mathbf{x} \in  \Upsilon_{N} (\mathbf{\widehat{B}}_v )}
\mu_{v, \mathbf{J}} ( \mathbf{x}) ,
$$
and for each $ \mathbf{x} \in  \Upsilon_{N} (\mathbf{\widehat{B}}_v
) $ there exists $  \mathbf{y} \in  \Upsilon_{N+1}
(\mathbf{\widehat{B}}_v )$ that is  more refined than $ \mathbf{x}
$. Therefore increasing $N$ the minimum in the previous formula
can only decrease, via Theorem \ref{pr}.

We conclude with a more explicit example. Let $d=2$, consider $
  \mathbf{{B}^*}_v$ (the sequence of balls
chosen in \cite{GLO}), suppose that (\ref{sommafin}) holds and take
$ N=1$. Note that $\mathbf{{B}^*}_v $ is less refined than each
sequence in $ \Upsilon_1 (\mathbf{{B}^* }_v )$. With simple
calculations we obtain
\begin{equation}\label{infinitomig}
\mu_{v , \mathbf{J}} ( \mathbf{ \widetilde{B} }_v ) =\mu_{v ,
\mathbf{J}} ( \mathbf{B^*}_v )-\sum_{i =1}^{3} ( |  B^*_v (1)   | -
|\widetilde{B}_v (i) | ) \lambda_{ v, \mathbf{J}  ,
\mathbf{\widetilde{B}}_v}(i) ,
\end{equation}
where $  \mathbf{ \widetilde{B} }_v  \in \Upsilon_1 (\mathbf{{B}^*
}_v )$.

 Notice that $  |  B^*_v (1)   |=5  $ and $ ( |  B^*_v (1)   |
- |\widetilde{B}_v (i) | )\geq 1 $ for $ i=1,2,3$.  Now one can take
$ \widetilde{B}_v (1), \widetilde{B}_v (2), \widetilde{B}_v (3),
\widetilde{B}_v (4)$ among the $4!$ possible choices selecting the
one which maximizes the sum in (\ref{infinitomig}). For some
interactions $\mathbf{J} \in \mathcal{J}$ the next inequalities  $
\mu_{v , \mathbf{J}} ( \mathbf{B^*}_v )>0$ and $\mu_{v , \mathbf{J}}
( \mathbf{ \widetilde{B} }_v ) <0 $ hold.
\end{rem}

\section{A general result on the extinction  of a
population}\label{sec4}

The following theorem gives a generalization of the extinction result on
Galton-Watson's process and it applies to processes that behave like
a supermartingale when they assume large values.

In the following theorem we will write for brevity of notation
$\mathbf{i}_h^k$ in place of the vector $(i_h,...,i_k)$, for $h\leq
k$. Furthermore, the equalities or inequalities between conditioned
probabilities have to be considered valid only if the conditioning
events have positive measure.

For each null event $ A$ we pose $
\mathbb{P}(\cdot | A) =1 $, in this way we can write the infimum in
place of the essential infimum.

\begin{theorem}\label{prop:1}
Let $\mathbf{X}=(X_{n}:n\in\mathbb{N})$ be a stochastic process over
$\mathbb{N}$. Suppose that there exists $N\in\mathbb{N}$ such that
the following relations hold:
\begin{itemize}
   \item[1)] $ \mathbb{P}(X_{n+1} =0 | X_n =0) =1$,  for  $n \in \mathbb{N}$;
 \item[2)] for $i \leq N$ there exists $ n_i \in \mathbb{N}_+$ such that
 $${q}_i ={\inf}_{m \in \mathbb{N},i_0,
\ldots ,i_{m-1}\in \mathbb{N}_+}
 \mathbb{P} ( X_{m+n_i }=0 |X_0=i_0, \ldots ,X_m=i_m)>0,\ i_m=i;$$
    \item[3)]
$\label{eq:f2}\mathbb{E}(X_{n+1}|X_0=i_0, \ldots ,X_n=i_n )\leq i_n$
a.s. for $ n \in \mathbb{N}$, $i_0, \ldots ,i_{n-1}\in \mathbb{N}$,
$i_n> N$;
    \item[4)] $
\label{eq:f3} p_{i}={\inf}_{m \in \mathbb{N},i_0, \ldots ,i_{m-1}\in
\mathbb{N}_+} \mathbb{P}(X_{m+1}\neq i|X_0=i_0, \ldots ,X_m=i_m)>0$,
$i_m=i>N$.
\end{itemize}
Then
$$
\lim_{n \to \infty}X_n=0\, \,a.s.
$$
\end{theorem}
\begin{proof}
Let $A=\{0,1,...,N\}$, $B=\{N+1,N+2,...\}$ where $N$ is given in the
theorem. Let us define
\begin{equation}\label{arresto1}\begin{array}{c}
                                  T_{A\rightarrow B}^{(1)}=\inf\{n\geq 0: X_n\in B\}, \quad
T_{B\rightarrow A}^{(1)}=\inf\{n>T_{A\rightarrow B}^{(1)} : X_n\in
A\}, \\
                                   T_{A\rightarrow B}^{(h)}=\inf\{n>T_{B\rightarrow A}^{{(h-1)}}:
X_n\in B\},\quad T_{B\rightarrow A}^{(h)}=\inf\{n>T_{A\rightarrow
B}^{(h)}: X_n\in A\}, \\
                                \end{array}
\end{equation}
for  $h\geq2$.

The random variables $ T_{A\rightarrow B}^{(h)}$, $T_{B\rightarrow
A}^{(h)}$, for  each $ h \geq 1 $, are stopping time. We put
$T_{A\rightarrow B}^{(h)}= \infty $ if the set, on which  the
infimum is defined, is empty or if $ T_{B\rightarrow A}^{(h-1)}
=\infty$. Similarly we write $T_{B\rightarrow A}^{(h)} =\infty$ if
the set, on which  the infimum is defined, is empty or if  $
T_{A\rightarrow B}^{(h)} =\infty$. The following inequalities are
obtained directly by definitions in (\ref{arresto1})
$$
T_{A\rightarrow B}^{(1)} \leq T_{B\rightarrow A}^{(1)} \leq
T_{A\rightarrow B}^{(2)} \leq \ldots \leq T_{A\rightarrow B}^{(h)}
\leq T_{B\rightarrow A}^{(h)} \leq \ldots
$$
The previous inequalities are strict until one of these stopping
times becomes infinite.

 Let us define the stopped process
 $ (Y^{(m)}_{n }=X_{n\wedge T_{B\rightarrow A}^{(m)}} : n \in \mathbb{N} )  $
 on $\{T_{A\rightarrow B}^{(m)}<\infty\}$, for
$m\in\mathbb{N}_+$. We do a partition of $\{T_{A\rightarrow
B}^{(m)}<\infty\}$ in the sets $ \{ \{T_{A\rightarrow B}^{(m)}=k\} :
k \in \mathbb{N}_+ \} $. On every set $ \{T_{A\rightarrow
B}^{(m)}=k\}$, the elements of $A$ are absorbing states  for
$Y_n^{(m)}$ when $n\geq k$, therefore $\{Y_n^{(m)}\}_{n \geq k}$ is
a non-negative supermartingale on $ \{T_{A\rightarrow B}^{(m)}=k\}$,
by hypothesis 3). Thus, see \cite{Wi}, there exists
\begin{equation}\label{time}
\lim_{n\rightarrow{+\infty}}Y_n^{(m)}<\infty  \,\,\, \hbox{ on }
\{T_{A\rightarrow B}^{(m)}<\infty\}  \,\,\,a.s.
\end{equation}
We will prove that the limit in (\ref{time}) belongs to $A$ almost
surely.

Given $k\in\mathbb{N}_+$, we prove (\ref{time}) on the set
$\{T_{A\rightarrow B}^{(m)}=k\}$. In fact if $i\in B$
$$
\mathbb{P}(\lim_{n\rightarrow{+\infty}}Y_n^{(m)}=i|T_{A\rightarrow
B}^{(m)}=k)=\mathbb{P}\bigg
(\bigcup_{h=k+1}^\infty\bigcap_{n=h}^\infty\{Y_n^{(m)}=i\}\bigg|T_{A\rightarrow
B}^{(m)}=k\bigg)
$$
$$
\leq
\sum_{h=k+1}^\infty\mathbb{P}\bigg(\bigcap_{n=h}^\infty\\
\{{Y_n^{(m)}=i}\}\bigg|T_{A\rightarrow B}^{(m)}=k \bigg) \leq
\sum_{h=k+1}^\infty\prod_{r=h+1}^\infty
\mathbb{P}(Y_r^{(m)}=i|Y_{h}^{(m)}=\ldots
=Y_{r-1}^{(m)}=i,T_{A\rightarrow B}^{(m)}=k)
$$
\begin{equation}\label{carnevale}
    =\sum_{h=k+1}^\infty\prod_{r=h+1}^\infty\mathbb{P}(X_r=i|X_{h}=
    \ldots=X_{r-1}=i,T_{A\rightarrow
B}^{(m)}=k) \,\,\,  ,
\end{equation}
where the last equality is a consequence of the fact that,
 if the limit belonged to $B$, then the process
 $(X_n)_{n\geq k}$ would never visit $ A $ and so, in this case, the
processes $ (Y^{(m)}_n)_{n\geq k}$ and $(X_n)_{n\geq k}$ would coincide.
Now, by using hypothesis 4) and a standard argument on the partition
of the trajectories, we obtain the following upper bound for
(\ref{carnevale})
\begin{equation}\label{vienezero}
    \sum_{h=k+1}^\infty\prod_{r=h}^\infty(1-p_i)=0.
\end{equation}
\comment{ In fact if $i\in B$, $r\geq \bar n$, denoting the event
$\{X_{\bar{n}}=\ldots=X_{r-1}=i\}$ by $J_{\bar n,r,i}$ and the set
of trajectories $L_{\bar
n,k,m}=\{{\mathbf{i}_0^{\bar{n}-1}}\in\mathbb{N}^{\bar
n}:\{\mathbf{X}_0^{\bar{n}-1}
=\mathbf{i}_0^{\bar{n}-1}\}\subset\{T_{A\rightarrow B}^{(m)}=k\}\}$
then almost surely
$$
\mathbb{P}(X_r\neq i|J_{\bar n,r,i},T_{A\rightarrow B}^{(m)}=k)=
$$
$$
=\sum_{{\mathbf{i}_0^{\bar{n}-1}}\in L_{n,k,m}}\mathbb{P} (X_r\neq
i|\mathbf{X}_0^{\bar{n}-1} =\mathbf{i}_0^{\bar{n}-1},J_{\bar
n,r,i},T_{A\rightarrow B}^{(m)}=k)
\mathbb{P}(\mathbf{X}_0^{\bar{n}-1}=\mathbf{i}_0^{\bar{n}-1}|J_{\bar
n,r,i},T_{A\rightarrow B}^{(m)}=k)\geq
$$
$$
\geq \inf_{{\mathbf{i}_0^{\bar{n}-1}}\in\mathbb{N}^{\bar
n}}\mathbb{P}(X_r\neq
i|\mathbf{X}_0^{\bar{n}-1}=\mathbf{i}_0^{\bar{n}-1},J_{\bar
n,r,i})\sum_{{\mathbf{i}_0^{\bar{n}-1}}\in L_{n,k,m}}
\mathbb{P}(\mathbf{X}_0^{\bar{n}-1}=\mathbf{i}_0^{\bar{n}-1}|J_{\bar
n,r,i},T_{A\rightarrow B}^{(m)}=k)\geq p_i.
$$
}
Hence we get that
\begin{equation*}
\lim_{n\rightarrow{+\infty}}Y_n^{(m)}\in A\ a.s.
\end{equation*}
or equivalently that
$$\mathbb{P}\left (\{T_{A\rightarrow B}^{(m)}<
\infty\} \setminus \{T_{B\rightarrow A}^{(m)}<\infty\} \right )=0 ,
$$
from which
\begin{equation}\label{sla}
    \mathbb{P}(\cdot |T_{A\rightarrow
B}^{({m-1})}<\infty)= \mathbb{P}(\cdot |T_{B\rightarrow
A}^{({m-1})}<\infty).
\end{equation}

Notice that, if the numbers $n_i $, for $i =0, \ldots , N$, verify
hypothesis 2) of the theorem, then, by taking  $n \geq \max\{n_i: i
\leq N\} $,  condition 2) is still verified. In fact, if the process
visits the state zero, then it indefinitely remains in zero, which
directly follows by hypothesis 1). Therefore let us define
$\widetilde{n}=\max \{ n_i: i \leq N \} \in \mathbb{N}_+$, then
hypothesis 2) is satisfied by using $ \widetilde{n}$ instead of $n_i
$ where the values of the
 $q_i$'s can only increase by replacing all the $n_i$'s with
$\widetilde{n}$. Hence all the $q_i$'s calculated setting $n_i=
\widetilde{n}$ are greater than some positive constant~$q$ which
 can be chosen equal to $ \inf\{ q_i : i =1, \ldots , N\}$.

Then we get, by (\ref{sla}), that for $ k \in \mathbb{N}_+$
\begin{equation*}\label{serve}
    \mathbb{P} ( T_{A\rightarrow B}^{((k+1) \widetilde{n}) } = \infty |
 T_{A\rightarrow B}^{(k \widetilde{n}) } <\infty)=
 \mathbb{P} ( T_{A\rightarrow B}^{((k+1) \widetilde{n}) } = \infty |
 T_{B\rightarrow A}^{(k \widetilde{n}) } <\infty).
\end{equation*}
 By denoting the set of trajectories
$M_{n,k}=\{{\mathbf{i}_0^n}\in\mathbb{N}^n:\{\mathbf{X}_0^n
=\mathbf{i}_0^n\}\subset\{T_{B\rightarrow A}^{(k \widetilde n)
}=n\}\}$, from the previous relation
we obtain
$$
\mathbb{P} ( T_{A\rightarrow B}^{((k+1) \widetilde{n}) } = \infty |
 T_{B\rightarrow A}^{(k \widetilde{n}) } <\infty)
 $$
$$
 =\sum_{n=1}^\infty
 \sum_{\mathbf{i}_0^n\in M_{n,k}}\mathbb{P}
 (T_{A\rightarrow B}^{((k+1) \widetilde{n}) } = \infty | T_{B\rightarrow A}^{(k \widetilde{n}) }=n,
 \mathbf{X}_0^n =\mathbf{i}_0^n) \mathbb{P} ( T_{B\rightarrow A}^{(k \widetilde{n}) }=n,
 \mathbf{X}_0^n =\mathbf{i}_0^n |  T_{B\rightarrow A}^{(k \widetilde{n}) } <\infty
 )
$$
$$
\geq \sum_{n=1}^\infty
 \sum_{\mathbf{i}_0^n\in M_{n,k}} \mathbb{P }( X_{n +\widetilde{n}}=0|
 \mathbf{X}_0^n =\mathbf{i}_0^n) \mathbb{P} ( T_{B\rightarrow A}^{(k \widetilde{n}) }=n,
 \mathbf{X}_0^n =\mathbf{i}_0^n |  T_{B\rightarrow A}^{(k \widetilde{n}) } <\infty
 )\geq q >0.
$$

Thus indicating $m =  \lfloor n /\widetilde{n} \rfloor $ for a
generic $n \in \mathbb{N}_+$, we obtain the following relation
$$
\mathbb{P}(T_{A\rightarrow B}^{(n)}<\infty)\leq
\prod_{k=2}^m\mathbb{P}(T_{A\rightarrow B}^{(k
\widetilde{n})}<\infty|T_{A\rightarrow B}^{((k-1)\widetilde{n})}<
\infty)\leq(1-q)^{m-1} .
$$
Since, for each $n \in \mathbb{N_+}$, $\{T_{A\rightarrow
B}^{(n)}<\infty\}\supset \{T_{A\rightarrow B}^{(n+1)}<\infty\}$, by
the monotone convergence theorem
\[
\mathbb{P}\bigg(\bigcap_{n=1}^\infty \{T_{A\rightarrow
B}^{(n)}<\infty\}\bigg)=
\lim_{n\rightarrow{+\infty}}\mathbb{P}(T_{A\rightarrow
B}^{(n)}<\infty)\leq \lim_{n\rightarrow{+\infty}} (1-q)^{  \lfloor n
/\widetilde{n} \rfloor -1} =0.
\]
Hence almost surely there exists a finite random index $S=2, 3,
\ldots $ such that $T_{A\rightarrow B}^{(S-1)}<\infty$,
$T_{B\rightarrow A}^{(S-1)}<\infty$ and $T_{A\rightarrow
B}^{(S)}=\infty$, then $X_n\in A$ for any $n\geq T_{B\rightarrow
A}^{(S-1)}$. It remains to show that the process can not stay
indefinitely in  $\{ 1, 2, \ldots, N \}$.

Let us define
$$
\widetilde{X}_k = X_{k \widetilde{n}}, \hbox{ for } k \in
\mathbb{N}.
$$
Note that for the process $\mathbf{\widetilde{X}} =(\widetilde{X}_n
: n \in \mathbb{N})$ there exists a random time almost surely finite
$$
\widetilde{T}_A = \inf \{n : \widetilde{X}_k\in A, \hbox{ for } k
\geq n \},
$$
such that the process remains indefinitely in $A$ after $
\widetilde{T}_A$. Moreover observe that $\widetilde{T}_A$ is not a
stopping time and it shall be taken into account the information
provided by the value of $\widetilde{T}_A$. Directly from hypothesis
2) it follows that
$$
\tilde{q}={\inf}_{m \in \mathbb{N},i_0 ,i_1, \ldots, i_{m-1}\in
\mathbb{N} ,i_m \in A} \mathbb{P} (\widetilde{X}_{m+1 }=0
|\mathbf{\widetilde{X}}_0^m =\mathbf{i}_0^m)
$$
is positive.

Now we will show that for each $n\in\mathbb{N}_+$,
\begin{equation*}\label{ess}
{\inf}_{m \geq n, \mathbf{i}_0^{n-2}\in \mathbb{N}^{n-1}, i_{n-1}\in
B, i_n, \ldots ,i_m\in A} \mathbb{P} ( \widetilde{X}_{m+1 }=0
|\widetilde{T}_A=n, \mathbf{\widetilde{X}}_0^m =\mathbf{i}_0^m)\geq
\tilde{q} >0.
\end{equation*}
We notice that for $\mathbf{i}_0^{n-2}\in \mathbb{N}^{n-1}$,
$i_{n-1}\in B$, $i_n, \ldots ,i_m\in A$,
 $$
\{\mathbf{\widetilde{X}}_0^m =\mathbf{i}_0^m, \widetilde{X}_{m+1 }=0
\} \subset\{\widetilde{T}_A=n  \},
 $$
from which
$$
\mathbb{P}(\mathbf{\widetilde{X}}_0^m =\mathbf{i}_0^m,
\widetilde{X}_{m+1 }=0)\leq\mathbb{P}(\widetilde{T}_A=n).
$$
Hence
$$\mathbb{P}(\widetilde{X}_{m+1 }=0 |\widetilde{T}_A=n,\mathbf{\widetilde{X}}_0^m =\mathbf{i}_0^m)=
\frac{\mathbb{P}(\widetilde{T}_A=n,\mathbf{\widetilde{X}}_0^m
=\mathbf{i}_0^m,\widetilde{X}_{m+1 }=0)}
{\mathbb{P}(\widetilde{T}_A=n,\mathbf{\widetilde{X}}_0^m
=\mathbf{i}_0^m)}
$$
$$=\frac{\mathbb{P} (\mathbf{\widetilde{X}}_0^m =\mathbf{i}_0^m,\widetilde{X}_{m+1 }=0)}
{\mathbb{P} (\widetilde{T}_A=n,\mathbf{\widetilde{X}}_0^m =\mathbf{i}_0^m)}\geq
\frac{\mathbb{P} (\mathbf{\widetilde{X}}_0^m =\mathbf{i}_0^m,\widetilde{X}_{m+1 }=0)}
{\mathbb{P}(\mathbf{\widetilde{X}}_0^m =\mathbf{i}_0^m)}=
\mathbb{P}(\widetilde{X}_{m+1 }=0|\mathbf{\widetilde{X}}_0^m =\mathbf{i}_0^m).
$$
From which by taking the infimum,
\begin{equation*}\label{ffii1}
    {\inf}_{m \geq n, \mathbf{i}_0^{n-2}\in \mathbb{N}^{n-1},
i_{n-1}\in B, i_n, \ldots ,i_m\in A} \mathbb{P} ( \widetilde{X}_{m+1
}=0 |\widetilde{T}_A=n, \mathbf{\widetilde{X}}_0^m =\mathbf{i}_0^m)
\end{equation*}
\begin{equation*}\label{ffii2}
    \geq{\inf}_{m \geq n, \mathbf{i}_0^{n-2}\in
\mathbb{N}^{n-1}, i_{n-1}\in B, i_n, \ldots ,i_m\in A}
\mathbb{P}(\widetilde{X}_{m+1 }=0|\mathbf{\widetilde{X}}_0^m
=\mathbf{i}_0^m)
\end{equation*}
\begin{equation*}\label{ffii3}
    \geq{\inf}_{m \in \mathbb{N},i_0 ,i_1, \ldots, i_{m-1}\in
\mathbb{N} ,i_m \in A} \mathbb{P} (\widetilde{X}_{m+1 }=0
|\mathbf{\widetilde{X}}_0^m =\mathbf{i}_0^m)=\tilde{q}>0.
\end{equation*}
Analogously to (\ref{vienezero}), by the latter inequalities and
standard arguments on the partition of trajectories, one obtains that
 the process $ \mathbf{\widetilde{X}}$
is eventually equal to zero. Obviously the same property is
obtained for the original process $\mathbf{X}$, i.e.
$\lim_{n\rightarrow{+\infty}}X_n=0\,\,a.s.$
\end{proof}

\begin{rem}\label{biologia}
We note that, in the previous theorem, the process $(X_n)_{n \in
\mathbb{N}}$ could be a nonhomogeneous Markov chain. In particular,
one can consider a culture of bacteria in which the number of its
population affects the ability of reproduction of the bacteria by
changing the probability that the cell dies before its mitosis. In
some way we can think that a process $(X_n)_{n \in \mathbb{N}}$,
verifying the assumptions of Theorem \ref{prop:1}, can be chosen as
a model for these biological cultures. Therefore the bacteria
cultures will die in a finite time.
\end{rem}

\section{ Applications of Theorem \ref{prop:1} to  perfect simulation} \label{sec5}

Let us consider a probability distribution $\psi_v$ indexed by $v\in
\mathbb{Z}^d$ and let $\sum_{l=0}^\infty \psi_v(l)=1 $. Moreover,
for each $v\in\mathbb{Z}^d$, let $\psi_{v}(0)>0 $.

Let us associate to each vertex $v \in \mathbb{Z}^d$ a sequence
$\mathbf{S}_v =(S_v(l)\Subset \mathbb{Z}^d :l\in\mathbb{N}_+) $ and
a mass $M_v $ such that
 $\inf_{v\in\mathbb{Z}^d}M_v \geq 1$.

Let $v \in \mathbb{Z}^d$ and $(D_n)_{n\in\mathbb{N}}$ be a
homogeneous Markov chain with countable state space
$\mathcal{C}=\{A\Subset\mathbb{Z}^d\}$.

At time zero the Markov chain has a initial measure $\nu^{(0)}$. The
rules of the dynamics are given in Section \ref{sec2}, it  only
needs to replace  $C_n$, $\mathbf{B}_v $, $ \lambda_v  $ with
$D_n$, $\mathbf{S}_v $, $ \psi_v  $ respectively.

\comment{ If $C_n = \emptyset $ then $D_{n+1 } = \emptyset $. Given
$D_n \neq \emptyset $ we proceeds to construct $D_{n+1 }$ as
follows: we select a random vertex $W_n\in C_n$ using
(\ref{sortver}) with constants $\{M_v\}_{v\in\mathbb{Z}^d}$
verifying $\inf_{v\in\mathbb{Z}^d}M_v \geq 1$, then we draw a random
value $ K_{w,n} $ with probability $\mathbb{P}(K_{w,n} =k) =
\psi_w(k) $, for $k \in \mathbb{N}$. If $ K_{w,n} =0 $ then $ C_{n+1
} = C_{n } \setminus \{ w\}$; if $ K_{w,n} =k $, for $ k\in
\mathbb{N}_+$, then $ C_{n+1 } = C_{n } \cup S_w(K_{w,n}) =C_{n }
\cup S_w(k) $. In particular if $ K_{w,n} =1$, then $C_{n+1 } = C_{n
}$ since $S_v(1)= \emptyset$ for each $v \in \mathbb{Z}^d$. }

Let us define, for each $v\in\mathbb{Z}^d$,
\begin{equation}\label{eta}
\eta_v=-\psi_v(0)+\sum_{l=1}^\infty|S_v(l)|\psi_v(l),
\end{equation}
which is similar to the birth-death expectation and  plays  the same
role.

 We are now in the position to present our result
on the extinction of the processes above defined.

\begin{corollary}\label{estinzione}
 Let $\eta_v$ as in (\ref{eta}), if
    $
\lim_{\Lambda\uparrow \mathbb{Z}^d} \sup_{v \notin
\Lambda}\eta_{v}<0 $,  then
  $  \limsup_{n \to \infty} D_n =~\emptyset$ almost surely.
\end{corollary}
\begin{proof}
Let $X_n = |D_n|$, we want to show that the process
$(X_n)_{n\in\mathbb{N}}$ verifies all the hypotheses of Theorem
\ref{prop:1}. Hypothesis 1) is trivially verified  because if $ D_n
= \emptyset$, then  $ D_{n+1} =\emptyset $ . We now verify
hypothesis 3). First of all note that from the assumption of the
corollary it follows the existence of a $ \delta>0 $ such that the
set
\begin{equation*}\label{delta}
R_\delta =\{v \in \mathbb{Z}^d:\eta_v > -\delta \}
\end{equation*}
has finite cardinality.

Fix $\delta > 0 $ such that $ | R_\delta | < \infty$, and
 define
$ a=\max\{0, M_v\eta_v: v \in R_\delta \} $. Consider $D_n \neq
\emptyset $, we easily see that
$$
\mathbb{E}( X_{n+1}| D_n )= \mathbb{E}( | D_{n+1} | \,\, | D_n )\leq
|D_n|+\sum_{ v\in D_n } \frac{M_v}{ \sum_{ u\in D_n }M_u }\eta_v.
$$
Under the assumption of the corollary and since $M_v\geq1$ for each
$v\in\mathbb{Z}^d$, we obtain
$$
\mathbb{E}( X_{n+1}| D_n ) \leq   |D_n|+ \frac{1}{\sum_{ u\in D_n
}M_u } \left [a |   R_\delta  | -\delta (|D_n|- | R_\delta  |)
\right ].
$$
We get that if
\begin{equation}\label{Ngrande}
X_n =|D_n | \geq  \left \lceil \frac{a  | R_\delta  | }{\delta} +  |
R_\delta
 |    \right  \rceil    \equiv N ,
\end{equation}
then $ \mathbb{E}( X_{n+1}| D_n )\leq X_n  $. Since
\begin{equation}\label{medcond}
\mathbb{E}( X_{n+1}| \mathbf{X}_0^n =\mathbf{i}_0^n ) = \sum_{
A\Subset \mathbb{Z}^d : |A| =i_n} \mathbb{E} ( X_{n+1}| D_n =A
)\mathbb{P} ( D_n =A | \mathbf{X}_0^n =\mathbf{i}_0^n ),
\end{equation}
we have that (\ref{medcond}) is lesser or equal to $ X_n=i_n$ when
$i_n \geq N$. Hence  hypothesis  3) is obtained by choosing $N$ as
in (\ref{Ngrande}), because all the summands in (\ref{medcond}) are
non-positive.

Now we show that
$$
\xi= \inf_{v \in \mathbb{Z}^d}\psi_v(0)>0.
$$
Note that
$$
\rho =\inf \{\psi_v (0): v\in R_\delta \} >0
$$
because it is an infimum on a finite set of positive numbers.
Moreover, from (\ref{eta}), it follows
$$
\rho' =\inf\{\psi_v (0): v\in R^c_\delta\} \geq \delta >0.
$$
Hence
$$
\xi  = \min \{ \rho,\rho' \}>0 .
$$
Therefore hypothesis 2) is verified for $n_i = N$ and the $q_i$'s
are larger or equal than  $ \xi^N>0 $, for $i \leq N$.

 We
also obtain 4) observing that $ p_i \geq \xi>0$ for each $ i \in
\mathbb{N}_+$.

Thus, from Theorem \ref{prop:1},
$$
\lim_{n\rightarrow{+\infty}}X_n=0\, \,a.s.
$$
There exists an almost surely finite random time $Y$ such that
$C_Y=\emptyset$.
\end{proof}

Given $ \mathbf{J} \in \mathcal{J} $, $v\in\mathbb{Z}^d$, $
\mathbf{B}_v \in \mathcal{B}_v $,  set
$$
S_v(l)=B_v(l)\setminus\{v\} \hbox{ for } l\in\mathbb{N}_+,
$$
and $ \psi_v =  \lambda_{ v,  \mathbf{J} , \mathbf{B}_v }  $, then,
by a simple calculation, $ \eta_v = \mu_{ v,  \mathbf{J}} (
\mathbf{B}_v ) $. Putting $ M_u =2\exp(\sum_{B,u\in B}|{J_B}|) $,
for each $ u \in \mathbb{Z}^d$, and $ \nu^{(0)} = \delta_{C_0} $ the
process $ (D_n )_n$ coincides with $ ( C_n )_n$ defined in Section
\ref{sec2}.

\begin{proof} [Proof of Theorem \ref{affonda}] The first part of Theorem \ref{affonda}  is a
direct consequence  of Corollary \ref{estinzione}.
\end{proof}

\medskip

We conclude the paper discussing an example in which an interaction
verifies hypothesis (H2) but does not verify (H1). The example is
constructed by using the property of universality described in
Theorem \ref{pr3}. Let $\mathbf{B}_v$'s be fixed, let us consider an
interaction $\mathbf{J}\in\mathcal{J}$ such that the inequality in
(H1) is verified. For a given $B_0\Subset\mathbb{Z}^d$ such that
$O\in B_0$, define $\mathbf{J}^{(L)}$ as
$$
J_{B}^{(L)}=\left\{\begin{array}{cc}
              J_B & \hbox{if }B\neq B_0; \\
              LJ_B & \hbox{if }B=B_0; \\
            \end{array}\right .
$$
where $L\in\mathbb{R}$. By elementary calculations, for a
sufficiently large $L>0$, it occurs that
$\mu_{O,\mathbf{J}^{(L)}}(\mathbf{B}_O)>0$, hence $\sup_{v \in
\mathbb{Z}^d} \mu_{ v, \mathbf{J}^{(L)}}(\mathbf{B}_v )>0$. Instead
$\lim_{\Lambda\uparrow \mathbb{Z}^d} \sup_{v \notin \Lambda} \mu_{
v, \mathbf{J}^{(L)}} ( \mathbf{B}_v )$ does not depend on $L$,
therefore it is less than zero.

Other examples, verifying (H2) but not (H1), can be naturally
constructed for each result in Section \ref{sec3} following the
scheme of the proofs and choosing suitable values for the
$\mathbf{J}$'s and the $\mathbf{B}_v$'s.

We notice that  Theorem \ref{prop:1}, by eliminating anyone of its
assumptions,  becomes false; examples can be easily constructed.

To finish we stress that condition (H2) differs from (H1) for two
reasons. First, the replacement of the supremum by the limit
superior improves the sufficient condition for the applicability of
the algorithm, but does not change the algorithm; second the
different choice of the sets $\mathbf{B}_v$'s improves the algorithm
and its applicability.

\appendix
\section{Algorithm for the infinite range Ising model} \label{App}

We present the algorithm for the infinite range Ising model showing
how to implement the result presented in Theorem \ref{ThIsing} in a
pseudo code. First one has to prove that, given the interaction
$\mathbf{J}\in \mathcal{J}_2$,
$$
\lim_{ \Lambda \uparrow \mathbb{Z}^d} \sup_{v \not \in
\Lambda}-2e^{-2\sum_{i=1}^\infty J^{(i)}_v}+e^{-\sum_{i=2}^\infty
J^{(i)}_v}+\sum_{l=2}^\infty l\left(e^{-\sum_{i=l+1}^\infty
J^{(i)}_v}-e^{-\sum_{i=l}^\infty J^{(i)}_v}\right)<0.
$$

If one does not use the finite range approximation presented in
\cite{GLO}, it is important that the sums $G_\mathbf{J}(v)=
\sum_{v'\in \mathbb{Z}^d \setminus v }|J_{\{v,v'\}}|$ are calculable
for each $ v \in \mathbb{Z}^d$ and that these values are given as
input in the algorithm. The easiest  case is the translational one
where $ J_{\{ u,v \} } =J_{\{ u+w,v +w\} } $, for any $ u,v, w \in
\mathbb{Z}^d$.

\medskip

\medskip
\noindent
$D$, $H$, $I$, $K$, $L$, $N$, $N_{STOP}$, $R$ are variables taking values in $\mathbb{N}$;\\
$U$ is a variable taking values in $[0,1]$;\\
$V$ is a variable taking values in $\mathbb{Z}^d$;\\
$Y$ is a variable taking values in $\{-1,1\}$;\\
$M$, $M_0$ are  variables taking values in  $\mathbb{R}$;\\
$C$,  $S$, $W$, $Z_-$, $Z_+$ are arrays of elements of $\mathbb{Z}^d$;\\
$Q$ is an array of elements of $\mathbb{Z}^d\times\mathbb{N}\times
\{E\Subset\mathbb{Z}^d\}^2$;\\
$G$ is a function from $\mathbb{Z}^d$ to $\mathbb{R}$;\\
$F$ is a function from $\mathbb{Z}^d \times \mathbb{N}$ to $[0,1]$;\\
$P$ is a function from $\mathbb{N}\times\mathbb{Z}^d
\times\{E\Subset\mathbb{Z}^d\}^2\times\{-1,1\}^{\mathbb{Z}^d}$ to $[0,1]$;\\
$T$ is  a bijective function from $\mathbb{N}$ to $\mathbb{Z}^d$;\\
$X$ is a function from $\mathbb{Z}^d$ to $\{-1,1\}\cup{\Delta}$
where $\Delta$ is an extra symbol that does not belong to $\{-1,1\}$
and it is called cemetery state;\\
RANDOM is a uniform random variable in $[0,1]$.

\medskip
\noindent
Algorithm 1: backward sketch procedure plus  construction  of optimal $ \mathbf{B}_v $'s \\
Input: $\mathbf{J}\in\mathcal{J}_2$; $C=(V_1,\ldots,V_{|C|})$;
$G(V)=
\sum_{V'\in \mathbb{Z}^d \setminus V }|J_{\{V,V'\}}|$;\\
Output: $N_{STOP}$; $Q$;\\
1. $N\leftarrow0$; ${N}_{STOP}\leftarrow0$; $Q\leftarrow\emptyset$; $D \leftarrow |C|$;\\
2. WHILE $C\neq\emptyset$\\
3. $N\leftarrow N+1$; $R\leftarrow 1$; $ M \leftarrow 0$; $ M_0 \leftarrow 0 $; $S\leftarrow \emptyset $;\\
4. $U\leftarrow \hbox{RANDOM}()$;\\
5. WHILE $\sum_{H=1}^R2\exp(G(V_H))/
\sum_{I=1}^{|C|}2\exp(G(V_I))<U$\\
6. $R\leftarrow R+1;$\\
7. END WHILE\\
8. $K\leftarrow0$;\\
9. $F(V_R,0)\leftarrow\exp(-2G(V_R))$;\\
10. WHILE $F(V_R,K)<U$\\
11. $K\leftarrow K+1$; $L\leftarrow1$;  \\
12. WHILE $G(V_R)-M-\sum_{I=1}^{L}|J_{\{V_R,T_I+V_R\}}|
\mathbf{1}(T_I \not \in S)>\\
\max\{|J_{\{V_R,T_I+V_R\}}|:I=1,\ldots,L,\;T_I\not \in  S \}$\\
13. $L\leftarrow L+1$;\\
14. $M_0 \leftarrow \max\{|J_{\{V_R,T_I+V_R\}}|:I=1,\ldots,L,\;T_I\not \in  S \}$;  \\
15. END WHILE\\
16. $A\leftarrow\min\{I=1,\ldots,L,\;T_I\not \in  S:|J_{\{V_R,T_I+V_R\}}|=M_0\}$;\\
17. $ M \leftarrow M +|J_{\{V_R,T_A+V_R\}}|$;\\
18. $S\leftarrow S\cup (T_A+V_R)$;\\
19. $W_K\leftarrow T_A+V_R$;\\
20. $F(V_R,K)\leftarrow\exp(-G(V_R)+\sum_{I=1}^K|J_{\{V_R,W_I\}}|)$;\\
21. END WHILE \\
22. IF $K=0$\\
23. $C\leftarrow C\setminus V_R$;\\
24. ELSE\\
25. FOR $I=1,\ldots,L$;\\
26. $C\leftarrow C\cup W_I$;\\
27. END FOR\\
28. END IF\\
29. $Q(N)\leftarrow(V_R,K,\bigcup_{I=1}^{L-1} W_I,\bigcup_{I=1}^L W_I)$;\\
30. END WHILE\\
31. $N_{STOP}\leftarrow N$;\\
32. RETURN $N_{STOP}$; $Q$.\\

\noindent Algorithm 2: forward spin assignment procedure\\
Input: $N_{STOP}$; $Q$;\\
Output: $\{X(V_1), \ldots,X(V_D) \}$;\\
33. $N\leftarrow N_{STOP}$;\\
34. $X(j)\leftarrow\Delta$ for all $j\in\mathbb{Z}^d$;\\
35. WHILE $N\geq1$\\
36. $(V,K,Z_-, Z_+)\leftarrow Q(N)$;\\
37. $U\leftarrow \hbox{RANDOM}()$;\\
38. IF $0\leq U\leq P_{V,Z_-, Z_+}^{[K]}(-X(V)|X)$\\
39. $Y=-1$;\\
40. ELSE $Y=1$;\\
41. END IF\\
42. $X(V)\leftarrow Y\cdot\mathbf{1}(K=0)+X(V)\cdot Y\cdot\mathbf{1}(K>0)$;\\
43. $N\leftarrow N-1$;\\
44. END WHILE\\
45. RETURN $\{X(V_1), \ldots,X(V_D) \}$.\\

\medskip

We write some comments  to facilitate the understanding of the
pseudo code.

\noindent Line 2. the b.s.p. ends when the set $C$ becomes empty.\\
\noindent Lines 5.-7. a random vertex $V_R$ in $C$ is chosen with
probability given in (\ref{sortver}). \\
\noindent Lines 10.-21. a random value $ K$, related to the vertex $
V_R$, is selected by Skorohod representation that uses $F_{V_R}(K)$
the cumulative distribution of $ \lambda$ (see (\ref{lambda})).
Notice that, for each $k$,
$ F_{V_R}(k)$ can be calculated with a finite number of elementary operations, when $G(v)$ is known. \\
\noindent Lines 12.-15. it is a small algorithm that finds for a
positive  sequence $ \{a_n\}_{n \in \mathbb{N} } $ with $ L =
\sum_{n\in \mathbb{N} } a_n < \infty$ the biggest element $ a_{\bar
n } = \max \{ a_n :n \in \mathbb{N}  \} $ and the index $\bar n$. We
stress that it is done in a finite number of steps. Iteratively  the
second biggest element is calculated and so on.  \\
\noindent Line 38. The probabilities $ p^{[k]}_{v, \mathbf{J},
\mathbf{B}_v}(-\sigma(v)|\sigma) $ defined in
(\ref{probp1})-(\ref{probp2}) depend on the finite sets $ B_v (k-1)$
and $B_v (k) $ that in the pseudo code are $Z_-$ and $Z_+$
respectively. In the pseudo code these probabilities are $P_{V,Z_-,
Z_+}^{[K]}(-X(V)|X)$. In (\ref{probp1}) and (\ref{probp2}) all the
sums have a finite number of elements, except one in (\ref{probp1})
that can be rewritten as $- G(v) +\sum_{u \in B_{v} (1)} |J_{\{
v,u\}}| $, which has a finite number of addenda.

\bibliography{versione7nov}

\begin{thebibliography}{GLO10}

\bibitem[CFF02]{CFF}
Francis Comets, Roberto Fern{\'a}ndez, and Pablo~A. Ferrari.
\newblock Processes with long memory: regenerative construction and perfect
  simulation.
\newblock {\em Ann. Appl. Probab.}, 12(3):921--943, 2002.

\bibitem[DSP08]{DP}
Emilio De~Santis and Mauro Piccioni.
\newblock Exact simulation for discrete time spin systems and unilateral
  fields.
\newblock {\em Methodol. Comput. Appl. Probab.}, 10(1):105--120, 2008.

\bibitem[GLO10]{GLO}
A.~Galves, E.~L{\"o}cherbach, and E.~Orlandi.
\newblock Perfect simulation of infinite range {G}ibbs measures and coupling
  with their finite range approximations.
\newblock {\em J. Stat. Phys.}, 138(1-3):476--495, 2010.

\bibitem[HS00]{HS00}
Olle H{\"a}ggstr{\"o}m and Jeffrey~E. Steif.
\newblock Propp-{W}ilson algorithms and finitary codings for high noise
  {M}arkov random fields.
\newblock {\em Combin. Probab. Comput.}, 9(5):425--439, 2000.

\bibitem[Lig85]{liggett}
T.M. Liggett.
\newblock {\em Interacting particle systems}, volume 276 of {\em Grundlehren
  der Mathematischen Wissenschaften [Fundamental Principles of Mathematical
  Sciences]}.
\newblock Springer-Verlag, New York, 1985.

\bibitem[MG98]{MG1998}
D.~J. Murdoch and P.~J. Green.
\newblock Exact sampling from a continuous state space.
\newblock {\em Scand. J. Statist.}, 25(3):483--502, 1998.

\bibitem[PW96]{PW}
James~Gary Propp and David~Bruce Wilson.
\newblock Exact sampling with coupled {M}arkov chains and applications to
  statistical mechanics.
\newblock {\em Random Structures Algorithms}, 9(1-2):223--252, 1996.

\bibitem[Wil91]{Wi}
David Williams.
\newblock {\em Probability with martingales}.
\newblock Cambridge University Press, Cambridge, 1991.

\end{thebibliography}
\end{document}